\documentclass[12pt,draftcls,onecolumn]{IEEEtran}

\usepackage[ansinew]{inputenc}
\usepackage{amsmath}
\usepackage{amssymb}
\usepackage{amsfonts}
\usepackage{amsthm}
\usepackage[english]{babel}
\usepackage[T1]{fontenc}
\usepackage{lmodern}
\usepackage{graphicx}
\usepackage{subfigure}
\newtheorem{theorem}{Theorem}
\newtheorem{proposition}{Proposition}

\newtheorem{corollary}{Corollary}
\newtheorem{lemma}{Lemma}
\theoremstyle{definition}
\newtheorem{remark}{Remark}
\newtheorem*{example}{Example}
\usepackage{color}

\newcommand{\guillemets}[1]{``#1''}

\begin{document}

\title{Global Analysis of a Continuum Model \\ for Monotone Pulse-Coupled Oscillators \footnote{The authors are grateful to Jean-Michel Coron for numerous valuable comments and useful discussions about the manuscript. This work was supported by the Belgian Network DYSCO (Dynamical Systems, Control, and Optimization), funded by the Interuniversity Attraction Poles Programme, initiated by the Belgian State, Science Policy Office, and A. Mauroy was supported as an FNRS Fellow (Belgian National Fund for Scientific Research).}}
\author{Alexandre Mauroy and Rodolphe Sepulchre}

\maketitle
\begin{abstract}
We consider a continuum of phase oscillators on the circle interacting through an impulsive instantaneous coupling. In contrast with previous studies on related pulse-coupled models, the stability results obtained in the continuum limit are global. For the nonlinear transport equation governing the evolution of the oscillators, we propose (under technical assumptions) a global Lyapunov function which is induced by a total variation distance between quantile densities. The monotone time evolution of the Lyapunov function completely characterizes the dichotomic behavior of the oscillators: either the oscillators converge in finite time to a synchronous state or they asymptotically converge to an asynchronous state uniformly spread on the circle. The results of the present paper apply to popular phase oscillators models (e.g. the well-known leaky integrate-and-fire model) and draw a strong parallel between the analysis of finite and infinite populations. In addition, they provide a novel approach for the (global) analysis of pulse-coupled oscillators.
\end{abstract}
\thanks{A. Mauroy is with the Department of Mechanical Engineering, University of California Santa Barbara, Santa Barbara, CA 93106, USA (e-mail: alex.mauroy@engr.ucsb.edu). The work was completed while the author was with the Department of Electrical Engineering and Computer Science, University of Liège, B-4000 Liège, Belgium .}\\
\thanks{R. Sepulchre is with the Department of Electrical Engineering and Computer Science, University of Liège, B-4000 Liège, Belgium (e-mail: r.sepulchre@ulg.ac.be).}

\section{Introduction}

Networks of interacting agents are omnipresent in natural \cite{Buck,Reynolds,Walker} as well as in artificial systems \cite{Hopfield,Paley,Strogatz_net}. In spite of their apparent simplicity, they may exhibit rich and complex ensemble behaviors \cite{Strogatz_book} and have led to intense research during the last few decades. In this context, coupled phase oscillators are generic models of paramount importance when studying the collective behaviors of a large collection of systems \cite{Winfree}. 

Phase oscillators appear as reductions or approximations of (realistic) dynamical oscillator models. They are obtained through the computation of a phase response curve (PRC) \cite{Smeal, Winfree}, which characterizes the phase sensitivity of an oscillator to an external perturbation, such as the influence of the neighboring oscillators in the network. Since phase oscillators are characterized by a one-dimensional state-space $S^1(0,2\pi)$, they are more amenable to a formal mathematical study of the collective behaviors even though the nonlinear interactions between oscillators often yield mathematical puzzles \cite{Kuramoto, Vicsek}.  

Within the network, oscillators interact through a nonlinear coupling. In most studied models, the coupling has a permanent influence on the network. However, in many situations encountered in biology or physics, the oscillators influence the network only during a tiny fraction of their cycle (e.g. yeast cell dynamics \cite{Boczko}). It is particularly so when the interconnection between the agents consists in the emission of fast pulses (spiking neurons \cite{Gerstner}, cardiac pacemaker cells \cite{Peskin}, earthquakes dynamics \cite{Olami}, etc.). In this paper, we consider the limit of an impulsive and instantaneous coupling, where a (pulse-coupled) oscillator interacts with the network only when its phase is equal to a given value. When considering the popular leaky integrate-and-fire oscillators, this model corresponds to Peskin's model \cite{Peskin, Kuramoto3}.

For a finite number of identical oscillators, previous studies show that the global behavior of pulse-coupled oscillators is dichotomic (\cite{Mauroy,Mirollo}): the oscillators converge either to a synchronized state or to an anti-synchronized state. In the present paper, we extend the result to infinite populations, showing that the global behavior of infinite populations is also dichotomic and thereby highlighting the perfect \emph{parallel} between finite and infinite populations.

Several earlier studies have provided \emph{local} stability results for infinite populations of pulse-coupled oscillators (see e.g. \cite{Abbott,Bressloff2,De_Smet,Kuramoto3,Vreeswijk}). In contrast, we present in this paper \emph{global} stability results for infinite populations of monotone oscillators (including leaky integrate-and-fire oscillators). To this end, we introduce a Lyapunov function which is induced by a $L^1$ norm and which has the interpretation of a \emph{total variation distance between (quantile) density functions}. Modulo technical conditions detailed in the paper, we show that the time evolution of the proposed Lyapunov function is governed by the derivative of the PRC, a result that leads to a global convergence analysis for monotone PRC's. 

Beyond the analysis of monotone pulse-coupled oscillators, the theory developed in the paper leads to general results on nonlinear partial differential equations (PDE). The analysis focuses on transport equations with a monotone dynamics (derived from a monotone PRC) and provides existence, uniqueness, and global stability results for the stationary solution of the PDE. In particular, the use of a total variation distance as a strict Lyapunov function seems novel and specific to the impulsive nature of the coupling (the results in \cite{Serre, Temple} suggest that total variation distance is of little use for systems of conservation laws). In this context, the result could potentially open new avenues to connect the monotonicity property of an infinite dimensional system to its stability properties.

The paper is organized as follows. In Section \ref{model}, we derive the transport PDE for the continuum of pulse-coupled oscillators from the original model of pulse-coupled integrate-and-fire oscillators (Peskin's model). Section \ref{dicho} presents numerical experiments showing that the continuum model is characterized by a dichotomic behavior. In addition, existence and uniqueness results are obtained for the stationary solution of the PDE. In Section \ref{main_result}, a strict Lyapunov function is proposed, which is inspired from our previous work \cite{Mauroy} on finite populations. In Section \ref{well_posedness}, we perform the convergence analysis of populations of monotone oscillators. The parallel between finite and infinite populations, as well as some extensions of the model, are discussed in Section \ref{discussion}. Finally, the paper closes with some concluding remarks in Section \ref{conclusion}.

\section{A phase density equation for pulse-coupled oscillators}
\label{model}

In this section, we introduce models of (monotone) integrate-and-fire oscillators with an (instantaneous) impulsive coupling. The oscillators are equivalent to phase oscillators and, in the continuum limit, evolve according to a phase density equation. The derivation of the corresponding nonlinear PDE is standard and similar developments are found in \cite{Abbott,Brown,Kuramoto3}.

\subsection{Pulse-coupled integrate-and-fire oscillators}

We consider models of integrate-and-fire oscillators \cite{Knight}. An integrate-and-fire oscillator is described by a scalar state variable $x$, which monotonically increases between the two thresholds $\underline{x}$ and $\overline{x}$ according to the dynamics $\dot{x}=F(x)$, $F>0$. When the oscillator reaches the upper threshold $\overline{x}$, it is reset to the lower threshold $\underline{x}$ (it is said to \textsl{fire}).

In \cite{Peskin}, Peskin proposed to study the behavior of $N$ integrate-and-fire oscillators interacting through an impulsive coupling. Whenever an oscillator fires, it sends out a pulse which instantaneously increments the state of all other oscillators by a constant value $K/N$, where $K$ is the coupling strength. The coupling is usually excitatory ($K>0$) but may also be inhibitory ($K<0$). The dynamics of a pulse-coupled integrate-and-fire oscillator $k\in\{1,\dots,N\}$ is then given by 
\begin{equation}
\label{stat_dyn2}
\dot{x}_k=F(x_k)+u_k(t)
\end{equation}
with the coupling
\begin{equation}
\label{coupling_discret}
u_k(t)=\frac{K}{N}  \sum_{\substack{j=1\\j\neq k}}^N \sum_{l=0}^{\infty} \delta(t-t^{(j)}_l)\,.
\end{equation}
The Dirac functions $\delta$ model the pulses which increment the state of oscillator $k$ at the firing times $t^{(j)}_l$, that is, when an oscillator $j\neq k$ fires.

Peskin's model was initially proposed with the popular \textsl{leaky} integrate-and-fire (LIF) oscillators, characterized by the affine vector field $F(x)=S-\gamma\, x$, $\gamma>0$ \cite{Lapicque}. However, the results of the present paper apply to \textsl{monotone oscillators}, that we define as integrate-and-fire oscillators with a monotone increasing or monotone decreasing vector field ($dF/dx>0$ or $dF/dx<0$). The class of monotone oscillators embraces a large variety of models, including the popular LIF model.

\subsection{Phase oscillators}

Integrate-and-fire oscillators are equivalently modeled as phase oscillators if the state dynamics \eqref{stat_dyn2} is turned into a phase dynamics. The phase \mbox{$\theta\in S^1(0,2\pi)$} is determined from the state $x\in[\underline{x},\overline{x}]$ by rescaling in such a way that $\theta=0$ corresponds to the low threshold $x=\underline{x}$ ---~the oscillator fires at phase $\theta=0$~--- and in such a way that a single (uncoupled) oscillator has a constant phase velocity $\dot{\theta}=\omega$, where $\omega$ is the natural frequency of the oscillator. This leads to the state-phase relation
\begin{equation}
\label{state_phase}
\theta=\omega\int_{\underline{x}}^x \frac{1}{F(s)}\, ds\,.
\end{equation}
Under the influence of the coupling $u_k(t)$, the state dynamics \eqref{stat_dyn2} corresponds to the phase dynamics
\begin{equation}
\label{phase_model}
\dot{\theta}_k=\omega+Z(\theta_k)\, u_k(t)\triangleq v(\theta_k,t)\,,
\end{equation}
where the function $Z\in C^1([0,2\pi])$ is the infinitesimal phase response curve (PRC) of the oscillator, that is, the phase sensitivity of the oscillator to an infinitesimal perturbation \cite{Izi_book, Smeal, Winfree}. For integrate-and-fire oscillators, the PRC has the closed-form expression (see \cite{Brown})
\begin{equation}
\label{PRC}
Z(\theta)=\frac{\omega}{F(x(\theta))}\,.
\end{equation}
It follows from \eqref{PRC} that monotone oscillators are characterized by a monotone PRC: \mbox{$dF/dx>0$} $\forall x\in[\underline{x},\overline{x}]$ leads to $Z'<0$ $\forall \theta \in[0,2\pi]$ and $dF/dx<0$ $\forall x\in[\underline{x},\overline{x}]$ leads to $Z'>0$ $\forall \theta \in[0,2\pi]$, where $Z'$ denotes the derivative of $Z$ with respect to the phase $\theta$.

\subsection{Phase density equation}

In the limit of a large number of $N\rightarrow\infty$ oscillators, the \textsl{infinite} population is a continuum characterized by a (nonnegative, continuous) phase density function 
\begin{equation}
\nonumber
\rho(\theta,t) \in C^0([0,2\pi]\times \mathbb{R}^+;\mathbb{R}^+)
\end{equation}
that satisfies the normalization
\begin{equation}
\nonumber
\int_0^{2\pi} \rho(\theta,t)\,d\theta=1 \quad \forall t\,.
\end{equation}
The quantity $\rho(\theta,t)d\theta$ is the fraction of oscillators with a phase between $\theta$ and $\theta+d\theta$ at time $t$. The time evolution of the density obeys the well-known continuity equation
\begin{equation}
\label{cont_equa}
\frac{\partial}{\partial t} \rho(\theta,t)=-\frac{\partial}{\partial \theta}\left[v(\theta,t)\,\rho(\theta,t)\right]\,,
\end{equation}
where the function 
\begin{equation}
\label{rho_J}
v(\theta,t)\,\rho(\theta,t)\triangleq J(\theta,t)
\end{equation}
is the (nonnegative, continuous) flux $J(\theta,t) \in C^0([0,2\pi] \times \mathbb{R}^+;\mathbb{R}^+)$. The quantity $J(\theta,t) dt$ represents the fraction of oscillators flowing through phase $\theta$ between time $t$ and $t+dt$. Since the phase $\theta$ is defined on $S^1(0,2\pi)\equiv\mathbb{R} \bmod 2\pi$, the flux must satisfy the boundary conditions
\begin{equation}
\label{CL}
J(0,t)=J(2\pi,t)\triangleq J_0(t)\quad \forall t\,.
\end{equation}
For the sake of simplicity, we use in the sequel the notation $J_0$ to denote the boundary flux (\ref{CL}). Since the oscillators fire at phase $\theta=0$, $J_0(t)$ is also called the \textsl{firing rate} of the oscillators.

\subsection{Continuous impulsive coupling}

The impulsive coupling originally defined for finite populations is extended to infinite populations as follows. Since the coupling strength $K/N$ is inversely proportional to the number of oscillators, the coupling does not increase as the number of oscillators grows: the constant $K$ corresponds to the net influence of the whole population through the coupling, regardless of the number of oscillators. In the limit $N\rightarrow \infty$, the firing of each oscillator of the infinite population produces an infinitesimal spike of size $K/N \rightarrow 0$.

The impulsive coupling is best expressed in terms of the flux $J_0$. For a finite population, the oscillators crossing $\theta=0$ at times $t^{(j)}_l$ induce a (discontinuous) flux \mbox{$J_0(t)=\frac 1 N \sum_j \sum_l \delta(t-t^{(j)}_l)$}. In the limit $N\rightarrow\infty$, the influence of a single oscillator is negligible, so that comparing the flux $J_0$ with the coupling \eqref{coupling_discret} yields the coupling
\begin{equation}
\label{cont_coupling}
u(t)=K\, J_0(t)\,.
\end{equation}

Roughly, for infinite populations, the impulsive coupling is identical for all the oscillators and \emph{proportional to the firing rate} $J_0$. In addition, the coupling is a continuous-time function interpreted as an infinite sum of infinitesimal spikes that result from the uninterrupted firings of the continuum.\\

With the continuous impulsive coupling \eqref{cont_coupling}, the phase dynamics \eqref{phase_model} is rewritten as
\begin{equation}
\label{dynamics}
\dot{\theta}=v(\theta,t)=\omega+Z(\theta)\,K\,J_0(t)\,.
\end{equation}
The final PDE for the continuum of pulse-coupled oscillators is derived as follows. At $\theta=0$, \eqref{rho_J} and \eqref{dynamics} yield the relationship
\begin{equation}
\label{equa_ajout}
J_0(t)=v(0,t)\,\rho(0,t)=\left[\omega+K\,Z(0)\,J_0(t)\right] \rho(0,t)
\end{equation}
and the flux at the boundary is explicitly given by
\begin{equation*}
J_0(t)=\frac{\omega\,\rho(0,t)}{1-K\,Z(0)\,\rho(0,t)}\,.
\end{equation*}
The continuity equation (\ref{cont_equa}) thus leads to the nonlinear PDE for the density
\begin{equation}
\label{equa_fin}
\frac{\partial \rho(\theta,t)}{\partial t}=-\omega \frac{\partial \rho(\theta,t)}{\partial \theta}-\frac{K\,\omega\,\rho(0,t)}{1-K\,Z(0)\,\rho(0,t)} \frac{\partial}{\partial \theta}\left[Z(\theta)\,\rho(\theta,t)\right]\,.
\end{equation}
The boundary condition (\ref{CL}) is expressed in terms of density as
\begin{equation}
\label{CL2}
\frac{\rho(0,t)}{1-K\,Z(0)\,\rho(0,t)}=\frac{\rho(2\pi,t)}{1-K\,Z(2\pi)\,\rho(2\pi,t)}\,\left(=\frac{J_0(t)}{\omega}\right)\,.
\end{equation}
The reader will notice that, in contrast to the periodicity condition on the flux, no periodicity is assumed on the density $\rho(\theta,t)$. [In particular, $\rho(0,t)\neq\rho(2\pi,t)$ if $Z(0)\neq Z(2\pi)$.]

The density of the pulse-coupled oscillators evolves according to the nonlinear PDE \eqref{equa_fin} with the boundary condition \eqref{CL2}. The PDE is studied in detail in the rest of the paper, with a particular attention to the case of a monotone PRC. Observe that a monotone PRC implies that $Z(0)\neq Z(2\pi)$.

\section{A dichotomic behavior}
\label{dicho}

Finite populations of pulse-coupled monotone oscillators exhibit a dichotomic behavior: they converge either toward a synchronized state or toward an anti-synchronized state (see \cite{Mauroy,Mirollo}). Similarly, infinite populations are characterized by a dichotomic asymptotic behavior, that depends on the coupling sign ($K>0$ or $K<0$) and on the derivative $Z'$ (or equivalently $dF/dx$). This remarkable behavior is described in the present section through numerical experiments and intuitive arguments, motivating the theoretical global analysis in the next sections. The actual proof of the dichotomic behavior is postponed to Section \ref{well_posedness}.

\subsection{Asymptotic behavior}
\label{subsec_asymp_simu}

Without coupling ($K=0$), the last term of (\ref{equa_fin}) disappears and the PDE is a standard transport equation. Its solution is a rigid translation of the initial density $\rho(\theta,0)=\rho_0(\theta)$ with a constant velocity $\omega$, that is, a traveling wave $\rho(\theta,t)=\rho_0((\theta-\omega t)\bmod 2\pi)$. In this case, any solution is periodic (with period $2\pi/\omega$) and the system is marginally stable.

When the oscillators are coupled, the last term of (\ref{equa_fin}) modifies the transport equation. Under the influence of the coupling, the velocity depends on both time and phase and the density is thereby \guillemets{stretched} or \guillemets{compressed}. This is illustrated when computing the total time derivative along a characteristic curve $\Lambda(t)$ defined by $\dot{\Lambda}=v(\Lambda(t),t)$, that is
\begin{equation}
\label{tot_der}
\frac{\partial \rho}{\partial t}+\frac{\partial \rho}{\partial \theta}v\big(\Lambda(t),t\big)=\frac{d\rho}{dt}=-\rho\big(\Lambda(t),t\big)\, J_0(t)\, K \,Z'\big(\Lambda(t)\big)\,,
\end{equation}
where (\ref{equa_fin}) and (\ref{CL2}) have been used. The total derivative shows that the density is modified on a characteristic curve whenever the PRC $Z$ is not constant. In addition to the rigid translation, the density undergoes a nonlinear transformation, possibly leading to asymptotic convergence to a particular density function corresponding to a particular stationary organization of the oscillators. 

The total derivative (\ref{tot_der}) gives clear insight that the sign of the derivative $K\, Z'$ is of primary importance. In fact, the sign of $K\, Z'$ will enforce a dichotomic behavior. The condition $K\,Z'(\theta)<0$ (or $K\,dF/dx>0$) will be shown to enforce convergence to a \emph{uniform flux} $J(\theta,t)=J^*$ on $S^1(0,2\pi)$. This situation, corresponding to the maximal spreading of the oscillators on the circle, is called the \emph{asynchronous} state \cite{Abbott,Vreeswijk} (Figure \ref{asynchro}). In contrast, the reverse condition $K\,Z'(\theta)>0$ (or $K\,dF/dx<0$) will be shown to enforce convergence to a \emph{delta-like flux} (Figure \ref{synchro}). This situation, characterized by the synchronization of all the oscillators, is the \emph{synchronous state}.
\begin{figure}[h]
\begin{center}
\subfigure[]{\includegraphics[width=0.4\textwidth]{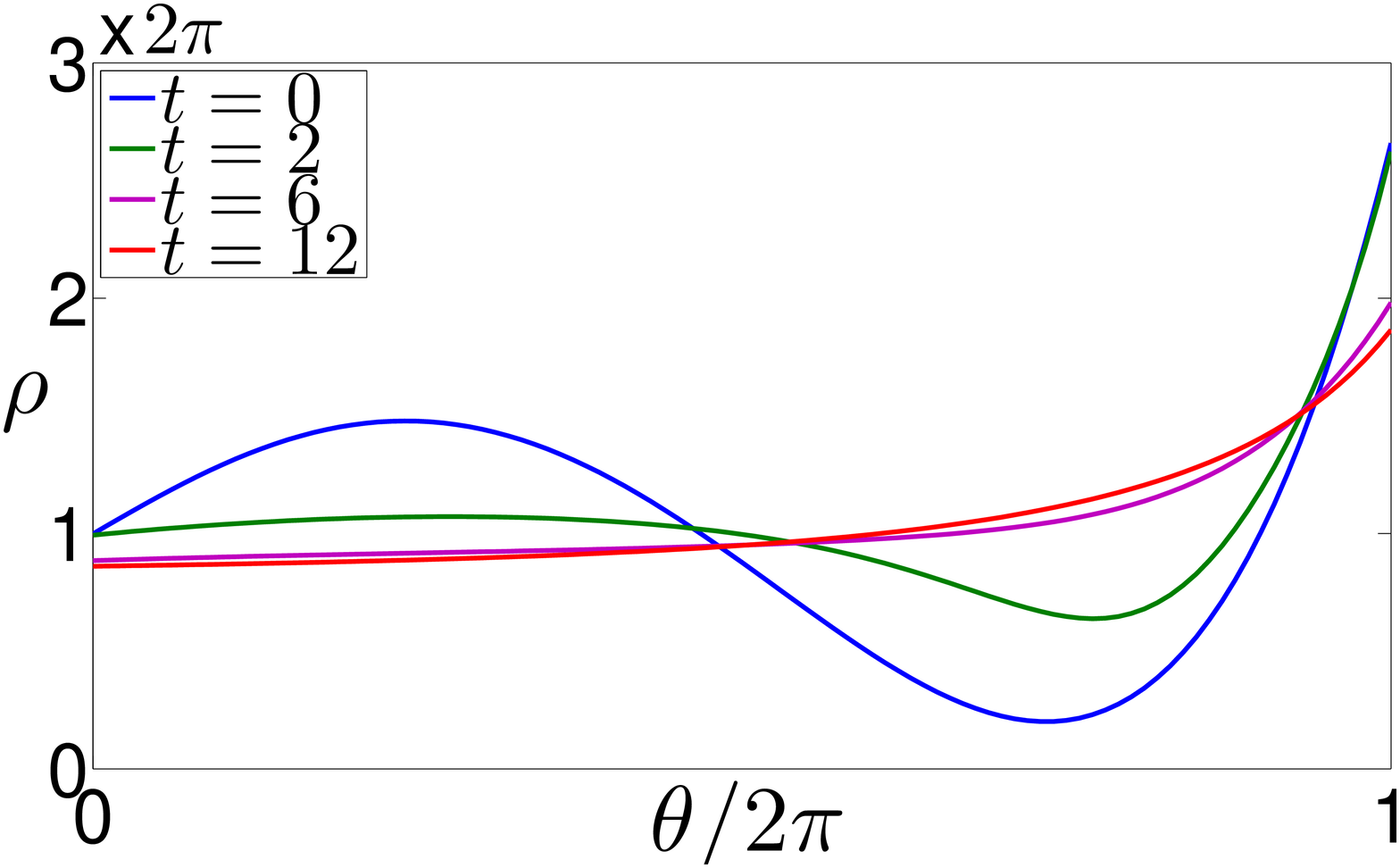}}
\subfigure[]{\includegraphics[width=0.4\textwidth]{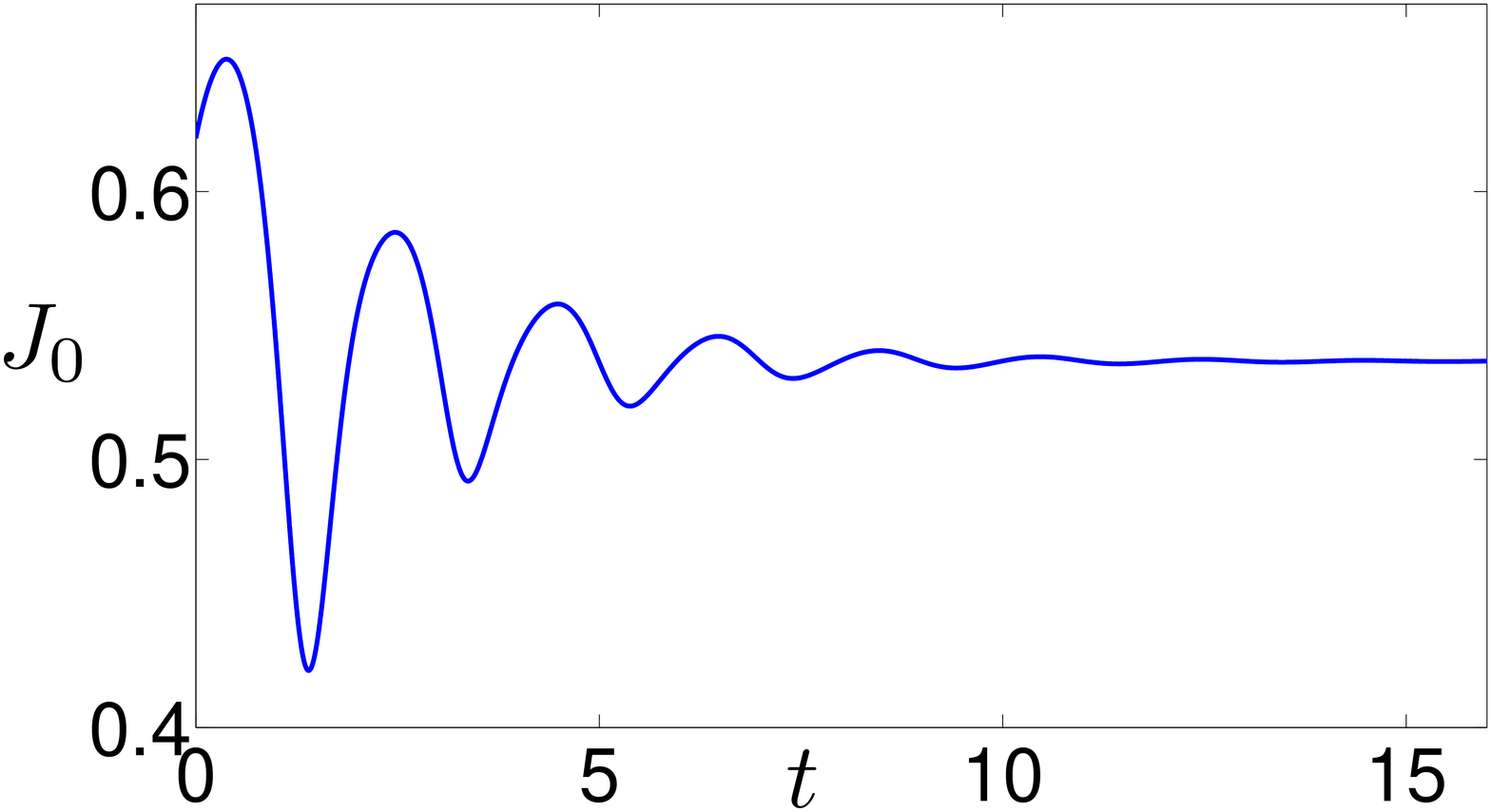}}
\caption{When $K\,Z'<0$ (or $K\,dF/dx>0$), the solution converges to the asynchronous state. With the monotone LIF dynamics $\dot{x}=2.1-2\,x$, $x\in[0,1]$ and with an inhibitory coupling $K=-0.1<0$, the function $K\,Z$ is monotone decreasing. (a) The density converges to a stationary solution $\rho^*$ and (b) the flux $J_0(t)$ tends to a constant value $J^*\approx 0.53$.}
\label{asynchro}
\end{center}
\end{figure}
\begin{figure}[h]
\begin{center}
\subfigure[]{\includegraphics[width=0.4\textwidth]{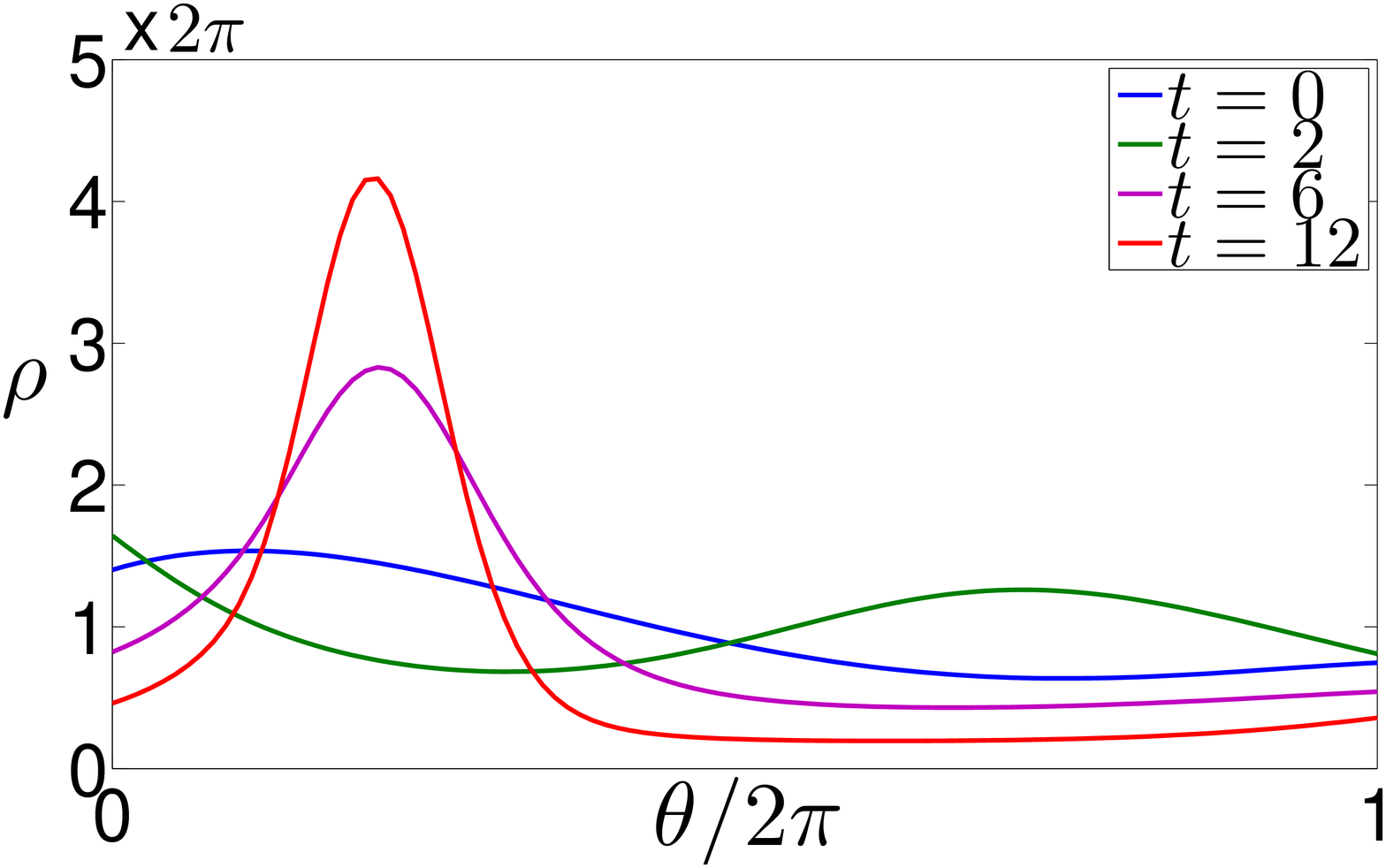}}
\subfigure[]{\includegraphics[width=0.4\textwidth]{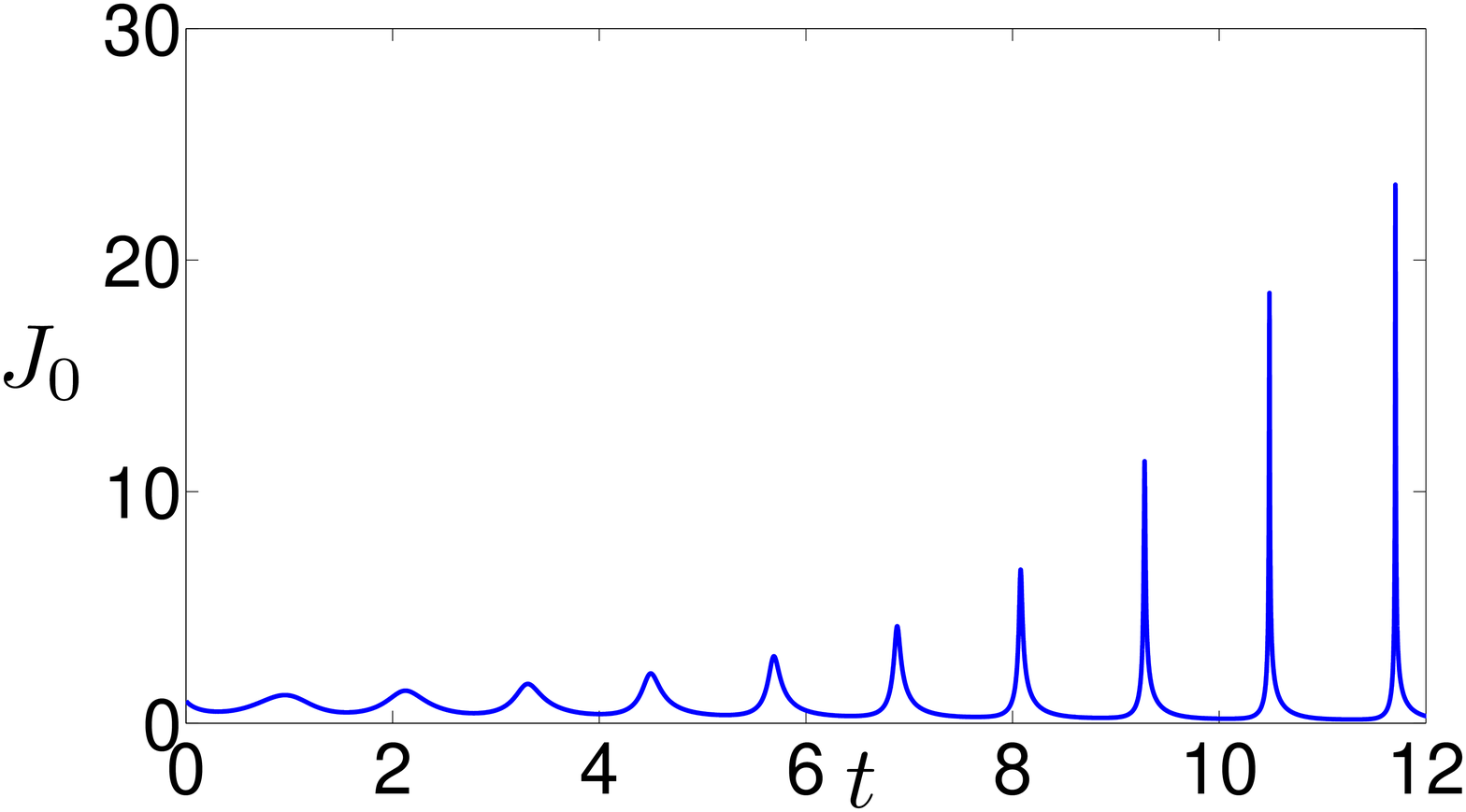}}
\caption{When $K\,Z'>0$ (or $K\,dF/dx<0$), the solution converges to the synchronous state. With the same LIF dynamics as in Fig. \ref{asynchro}, but with an excitatory coupling $K=0.1>0$, the function $K\,Z$ is monotone increasing. (a) The density converges to a synchronous solution and (b) the flux $J_0(t)$ tends to a Dirac function.}
\label{synchro}
\end{center}
\end{figure}

We remark that reversing the coupling sign ($K>0$ or $K<0$) has the same effect as reversing the monotonicity of $Z$ ($Z'<0$ or $Z'>0$).

The dichotomic asymptotic behavior is in agreement with the equivalent dichotomic behavior of finite populations: the synchronous and asynchronous states are the exact analog of the synchronous and \textsl{splay} states observed for finite populations. This is obvious in the case of synchronization, that occurs both for finite and infinite populations when $K\,dF/dx<0$ \cite{Mirollo}. Asynchronous state and splay state are also equivalent behaviors, that both occur when $K\,dF/dx>0$ \cite{Mauroy}. They are both anti-synchronized behaviors, for which the oscillators evenly spread over the circle $S^1(0,2\pi)$.


\subsection{Stationary asynchronous state}

We will now characterize the stationary solution of the PDE (\ref{cont_equa}), i.e. the asynchronous state corresponding to a constant flux $J(\theta,t)=J^*$. From (\ref{rho_J}), the stationary density must satisfy
\begin{equation}
\label{stat_density}
\rho^*(\theta)=\frac{J^*}{\omega+K\, Z(\theta)J^*}\,.
\end{equation}
The stationary asynchronous state thus exists if there exists a value $J^*>0$ so that the stationary density is nonnegative and bounded
\begin{equation}
\label{pos_density}
0\leq \frac{J^*}{\omega+K\, Z(\theta)J^*}<\infty \,,\quad \forall \theta \in [0,2\pi]\,,
\end{equation}
and normalized
\begin{equation}
\label{equa_E0}
\int_0^{2\pi}\frac{J^*}{\omega+K\, Z(\theta)J^*}d\theta=1\,.
\end{equation}
The condition (\ref{pos_density}) and the continuity of $Z$ imply that the stationary solution is continuous.

It is noticeable that for finite populations, the equivalent stationary splay state is a phase-locked configuration: at each firing, the $N$ oscillators are characterized by constant phases \mbox{$\mathbf{\Theta}^*=(\theta^*_1,\cdots,\theta^*_{N-1},\theta^*_N=2\pi$)} and fire at a constant rate. In particular, the value $\theta^*_{N-1}$ is related to the firing rate and is well-approximated by the stationary flux $J^*$ when $N\gg 1$.\\

The following proposition gives necessary and sufficient conditions to ensure the existence and uniqueness of a stationary solution $\rho^*\in C^0([0,2\pi];\mathbb{R}^+)$.
\begin{proposition}
\label{prop_E0}
A stationary flux $J^*>0$ satisfying the conditions (\ref{pos_density}) and (\ref{equa_E0}) exists if and only if the inequality
\begin{equation}
\label{cond2_E0}
\lim_{\begin{subarray}{c}s\rightarrow r \\
s>r \end{subarray}}\int_0^{2\pi}\frac{1}{K\, Z(\theta)+s}\,d\theta>1
\end{equation}
is satisfied with
\begin{equation}
\nonumber
r\triangleq\begin{cases}
0 & \textrm{if } K\, Z(\theta)\geq 0 \quad \forall \theta\in[0,2\pi]\, , \\
\left|\min_{\theta\in[0,2\pi]}\Big(K\, Z(\theta)\Big)\right| & \textrm{otherwise}\, .
\end{cases}
\end{equation}
Moreover, the solution is unique when it exists.
\end{proposition}
\begin{proof}
Inequality (\ref{pos_density}) implies that the velocity $\omega+K\, Z(\theta)J^*$ is strictly positive, so that
\begin{equation}
\nonumber
J^* \in \mathcal{J}\triangleq \left(0,\lim_{\substack{s\rightarrow r \\ s>r}}\frac{\omega}{s}\right)\,.
\end{equation}
The function
\begin{equation}
\nonumber
W(J)=\int_0^{2\pi}\frac{J}{\omega+K\, Z(\theta)J}\,d\theta
\end{equation}
satisfies $W(0)=0$, is continuous on $\mathcal{J}$, and is strictly increasing on $\mathcal{J}$ since
\begin{equation}
\nonumber
\frac{dW}{dJ}=\int_0^{2\pi}\frac{\omega}{\left[\omega+K\, Z(\theta)J\right]^2}\,d\theta>0 \quad \forall J\in \mathcal{J}\,.
\end{equation}
As a consequence, the equation $W(J)=1$, which is equivalent to condition (\ref{equa_E0}), has a (unique) solution $J^*\in\mathcal{J}$ if and only if
\begin{equation}
\nonumber
\lim_{\substack{s\rightarrow r \\ s>r}} W(\omega/s)=\lim_{\substack{s\rightarrow r \\ s>r}}\int_0^{2\pi}\frac{1}{K\, Z(\theta)+s}d\theta>1\,,
\end{equation}
which concludes the proof.
\end{proof}
Proposition \ref{prop_E0} implies that the coupling constant $K$ must be bounded. For integrate-and-fire oscillators, the bounds on the coupling constant are computed analytically and are given in the following corollary.
\begin{corollary}
\label{corol_stat_dens}
For integrate-and-fire dynamics $\dot{x}=F(x)$, a stationary flux $J^*>0$ fulfilling the conditions (\ref{pos_density}) and (\ref{equa_E0}) exists if the coupling constant satisfies
\begin{equation}
\label{bound_K}
\lim_{\substack{s\rightarrow F_\mathrm{min} \\ s<F_\mathrm{min}}}\int_{\underline{x}}^{\overline{x}} \frac{s}{s-F(x)} dx<K<\overline{x}-\underline{x} \,,
\end{equation}
with $F_\mathrm{min}=\min_{x\in[\underline{x},\overline{x}]}\big(F(x)\big)$.
\end{corollary}
\begin{proof}
If the coupling is excitatory ($K>0$), it follows from \eqref{PRC} that $K\,Z(\theta)\geq 0$ $\forall \theta\in[0,2\pi]$, so that $r=0$. Then, the condition (\ref{cond2_E0}) of Proposition \ref{prop_E0} can be rewritten as
\begin{equation}
\nonumber
\int_0^{2\pi} \frac{1}{K\, Z(\theta)} \,d\theta=\int_0^{2\pi} \frac{F(x(\theta))}{K\,\omega} d\theta= \int_{\underline{x}}^{\overline{x}}\frac{1}{K} \, dx>1
\end{equation}
given (\ref{state_phase}), or equivalently
\begin{equation}
\label{bound_K_max}
K<\overline{x}-\underline{x}\,.
\end{equation}
If the coupling is inhibitory ($K<0$), it follows from \eqref{PRC} that $r=-K\omega/F_\mathrm{min}$. Then, the condition (\ref{cond2_E0}) of Proposition \ref{prop_E0} is rewritten as
\begin{equation}
\nonumber
\lim_{\substack{s\rightarrow r \\ s>r}} \int_0^{2\pi} \frac{1}{K\,Z(\theta)+s} d\theta=\lim_{\substack{s\rightarrow F_\mathrm{min} \\ s<F_\mathrm{min}}} \int_0^{2\pi} \frac{1}{\frac{K\omega}{F(x(\theta))}-\frac{K\omega}{s}}d\theta= \lim_{\substack{s\rightarrow F_\mathrm{min} \\ s<F_\mathrm{min}}} \int_{\underline{x}}^{\overline{x}}\frac{1}{K} \frac{1}{1-\frac{F(x)}{s}} dx>1
\end{equation}
and we obtain the lower bound \eqref{bound_K} on $K$.
\end{proof}

Condition \eqref{bound_K_max} is easy to interpret in the case of finite populations: it is a necessary condition for the existence of a stationary phase-locked configuration. Since the average state difference between $N$ oscillators is $(\overline{x}-\underline{x})/N$, a coupling strength $K>\overline{x}-\underline{x}$ yields an \guillemets{avalanche} phenomenon igniting a chain reaction of firings. In this situation, a phase-locked behavior of $N$ distinct oscillators cannot exist.


\section{A strict Lyapunov function induced by the total variation distance}
\label{main_result}

Lyapunov analysis is a classical approach to study the stability of nonlinear PDE's (see e.g. \cite{Coron2008}). In this section, we extend our previous results obtained for finite populations \cite{Mauroy} to construct a strict Lyapunov function for the PDE \eqref{equa_fin}. The Lyapunov function, inspired by the $1$-norm introduced in \cite{Mauroy}, is a $L^1$ distance interpreted as the \emph{total variation distance between quantile densities}.

\subsection{Quantile density}

The description of the infinite population through the density $\rho(\theta)$ is not equivalent to the description of a finite population through a vector $\mathbf{\Theta}=(\theta_1 , \cdots , \theta_N)$. While the former corresponds to the \guillemets{amount} of oscillators as a function of the phase, the latter corresponds to the phase as a function of the oscillator index. To establish an equivalence between finite and infinite populations, we introduce an index for infinite populations of oscillators and use the concept of \textsl{quantile function}.

For infinite populations, the oscillators can be continuously labeled on the interval \mbox{$[0,1]$} and an oscillator index $\varphi \in [0,1]$  is defined as follows. Given a density function \mbox{$\rho: [0,2\pi]\mapsto \mathbb{R}^+$}, the cumulative density function $P(\theta):[0,2\pi]\mapsto [0,1]$, defined as
\begin{equation}
\nonumber
P(\theta)=\int_0^\theta \rho(s) \, ds\,,
\end{equation}
attributes an index $\varphi=P(\theta)\in [0,1]$ to each oscillator with phase $\theta$. In particular, an index $\varphi=0$ (resp. $\varphi=1$) is attributed to the oscillator at phase $\theta=0$ (resp. $\theta=2\pi$).

Next, to complete the equivalent description of infinite populations, we introduce the quantile function (widely used in statistics \cite{Parzen}): the quantile function $Q:[0,1] \mapsto [0,2\pi]$ is the inverse cumulative density function, that is,
\begin{equation}
\nonumber
Q(\varphi)=P^{-1}(\varphi)=\inf\{\theta|P(\theta)\geq \varphi\}\,.
\end{equation}
The (continuous) quantile function is equivalent to the (discrete) description $\mathbf{\Theta}$ of finite populations. For a finite number of $N$ distinct oscillators, at each firing of an oscillator ($\theta_N=2\pi)$, the remaining phases $\theta_k$ are the $N$-quantiles $\theta_k=Q^{(N)}_k$, that is 
\begin{equation}
\label{vector}
\mathbf{\Theta}=\left(\theta_1 , \cdots , \theta_{N-1}, 2\pi \right)=\left(Q^{(N)}_1 , \cdots , Q^{(N)}_{N-1} , 2\pi \right)\,.
\end{equation}
When the number of oscillators tends to infinity, the $N$-quantiles are replaced by the continuous quantile function $Q$. (Roughly speaking, the quantile function plays the role of the vector $\mathbf{\Theta}$ with an infinity of components.)

As the analog of the density $\rho$, the \textsl{quantile density function} \cite{Parzen}, also called \textsl{sparsity function}, is the function $q:[0,1]\mapsto\mathbb{R}^+$ that satisfies (see Figure \ref{CDF})
\begin{equation*}
Q(\varphi)=\int_0^{\varphi} q(s)\, ds\,.
\end{equation*}
The quantile density function, which is the derivative of the quantile function, expresses the increase of phase per unit increase of oscillator index. The density function is linked to the quantile density function by the relationship
\begin{equation}
\label{der_rho_comp}
q(\varphi)=\frac{dQ}{d\varphi}=\frac{1}{\rho\big(Q(\varphi)\big)}\,.
\end{equation}
In order to avoid some ill-defined cases, the condition $\rho>0$ must be satisfied on $[0,2\pi]$.

\begin{figure}[h]
\begin{center}
\includegraphics[width=0.9\textwidth]{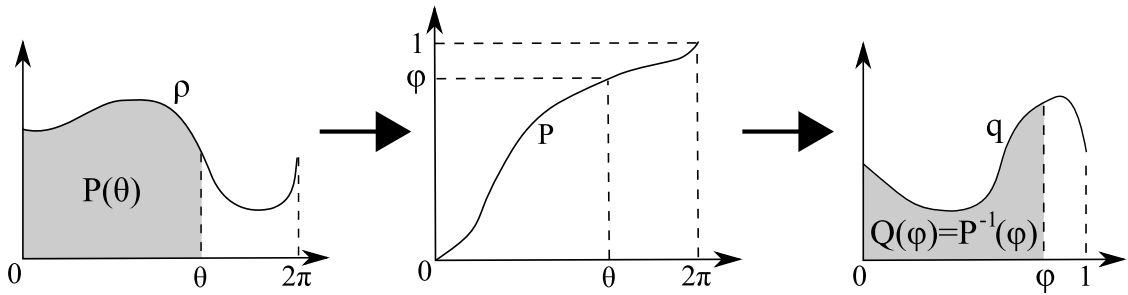}
\caption{The density function $\rho(\theta)$ (left) has a cumulative density $P(\theta)$ (center). The quantile function \mbox{$Q(\varphi)=P^{-1}(\varphi)$} is the cumulative density function of the quantile density function $q(\varphi)$ (right).}
\label{CDF}
\end{center}
\end{figure}

The reader will notice that, as the oscillators density $\rho(\theta,t)$ depends on time in the model (\ref{cont_equa}), the associated quantile function and quantile density function also depend on time and are then rigorously defined as the two-variable functions $Q(\varphi,t)$ and $q(\varphi,t)$. In addition, we denote the quantile function and the quantile density associated to the stationary solution (\ref{stat_density}) by $Q^*(\varphi)$ and $q^*(\varphi)$ respectively.

\subsection{Total variation distance}
\label{lyap_discr}

A $1$-norm introduced in our previous study \cite{Mauroy} on finite populations of monotone pulse-coupled oscillators leads to a $L^1$ distance (between quantile densities) in the case of infinite populations, a distance which can be interpreted as a \textsl{total variation distance}.

For a finite population of $N$ distinct oscillators, it is sufficient to consider only the firing instants, for which $\theta_N=2\pi$. Then, the corresponding discrete system is characterized by a simple Lyapunov function: the distance (induced by a $1$-norm) between a configuration $\mathbf{\Theta}$ and the stationary phase-locked configuration $\mathbf{\Theta^*}$. It is expressed as
\begin{equation*}
\mathcal{V}^{(N)}=\left|\theta_1-\theta^*_1\right|+\sum_{k=1}^{N-2}\left|\left(\theta_k-\theta_{k+1}\right)-\left(\theta^*_{k}-\theta^*_{k+1}\right)\right|+\left|\theta_{N-1}-\theta^*_{N-1}\right|\,.
\end{equation*}
The phases $\theta_k$ can be replaced by the $N$-quantiles, according to \eqref{vector}, and one obtains
\begin{equation}
\label{TVD_discrete}
\mathcal{V}^{(N)}=\left|Q^{(N)}_1-Q^{*(N)}_1\right|+\sum_{k=1}^{N-2}\left|\left(Q^{(N)}_k-Q^{(N)}_{k+1}\right)-\left(Q^{*(N)}_{k}-Q^{*(N)}_{k+1}\right)\right|+\left|Q^{(N)}_{N-1}-Q^{*(N)}_{N-1}\right|\,.
\end{equation}
In the limit $N\rightarrow \infty$, the continuous equivalent of \eqref{TVD_discrete} corresponds to the $L^1$ distance between the quantile density functions:
\begin{equation}
\label{Lyap_tot}
\mathcal{V}(\rho)=\int_0^1\left|\frac{\partial Q}{\partial \varphi}-\frac{dQ^*}{d\varphi}\right|d\varphi=\left\|q-q^*\right\|_{L^1} \quad \forall \rho  \in C^0([0,2\pi]\times \mathbb{R}^+;\mathbb{R}^+_0)\,.
\end{equation}
The second equality is obtained through \eqref{der_rho_comp}. One verifies that $\mathcal{V}(\rho)=0$ $\Leftrightarrow$ $q=q^*$ a.e. $\Leftrightarrow$ $\rho=\rho^*$ a.e.

Our previous study \cite{Mauroy} shows that, under mild conditions, quantity \eqref{TVD_discrete} decreases at the successive firings of the oscillators, enforcing a contraction property for the $1$-norm. We claim that, for infinite populations, the continuous equivalent \eqref{Lyap_tot} also decreases at the successive firings of the continuum, that is, \eqref{Lyap_tot} decreases continuously with time. The main result of this paper will thus establish \eqref{Lyap_tot} as a good Lyapunov function for the PDE (\ref{equa_fin}).

The Lyapunov function \eqref{Lyap_tot} is interpreted as a total variation distance. Indeed, the total variation distance between two random variables corresponds to the $L^1$ distance between the corresponding density functions (see \cite{Dunford} for further details). In the present case, the total variation distance is the sum of the maximum differences between the two quantile functions $Q$ and $Q^*$, minus the sum of the minimum differences (Figure \ref{distance}).
\begin{figure}[h]
\begin{center}
\includegraphics[width=6cm]{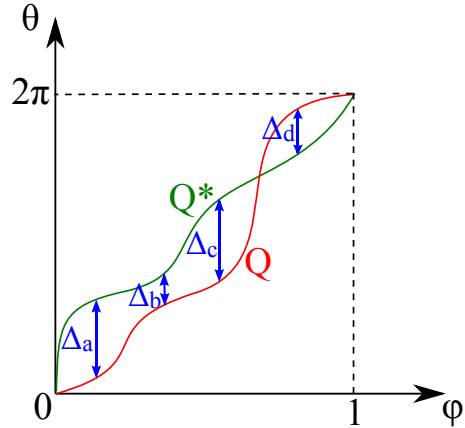}
\caption{The Lyapunov function (\ref{Lyap_tot}) is the total variation distance between two quantile density functions. In the example of the figure, the distance is equal to $\mathcal{V}=2(\Delta_a+\Delta_c+\Delta_d)-2\Delta_b$, with $\Delta_{a,b,c,d}>0$.}
\label{distance}
\end{center}
\end{figure}

\subsection{Time evolution of the Lyapunov function}

Our main result shows that the candidate Lyapunov function \eqref{Lyap_tot} has a monotone time evolution provided that the PRC is monotone.
\begin{theorem}
\label{theo_Lyap_evol}
Let $\rho(\theta,t)\in C^0([0,2\pi]\times \mathbb{R}^+;\mathbb{R}^+_0)$, with $\|\rho\|_{L_1}=1$, be a strictly positive solution of (\ref{equa_fin})-(\ref{CL2}). If the stationary density (\ref{stat_density}) exists and if either $Z''(\theta)\geq0$ $\forall \theta\in[0,2\pi]$ or $Z''(\theta)\leq0$ $\forall \theta\in[0,2\pi]$, then the Lyapunov function (\ref{Lyap_tot}) satisfies
\begin{equation}
\label{mono_V}
J(0,t) \min_{\theta\in[0,2\pi]} \Big(K\,Z'(\theta)\Big) \mathcal{V}(\rho) \leq \dot{\mathcal{V}}(\rho) \leq J(0,t) \max_{\theta\in[0,2\pi]} \Big(K\,Z'(\theta)\Big) \mathcal{V}(\rho)\,.
\end{equation}
\end{theorem}
\begin{proof}
The time derivative of \eqref{Lyap_tot} can be written as
\begin{equation}
\label{equa_dot_V}
\dot{\mathcal{V}}(\rho)=\int_0^1 \textrm{sign} \left(\frac{\partial Q}{\partial \varphi}-\frac{dQ^*}{d\varphi}\right) \frac{\partial}{\partial t}\left(\frac{\partial Q}{\partial \varphi}-\frac{dQ^*}{d\varphi}\right) d\varphi
 =\int_0^1 \textrm{sign} \big(G(\varphi,t)\big) \frac{\partial}{\partial \varphi}\left(\frac{\partial Q}{\partial t}\right) d\varphi\,,
\end{equation}
with $G(\varphi,t)=\partial Q/\partial \varphi-dQ^*/d\varphi$ and with the signum function $\textrm{sign}(x)=x/|x|$. 

Differentiating the expression $\theta \equiv Q(P(\theta,t),t)$ with respect to time $t$ leads to
\begin{equation}
\nonumber
0=\frac{d}{dt}[Q(P(\theta,t),t)]=\frac{\partial Q}{\partial t}(P(\theta,t),t)+\frac{\partial Q}{\partial \varphi}(P(\theta,t),t) \frac{\partial P}{\partial t}(\theta,t)
\end{equation}
or
\begin{equation}
\nonumber
\frac{\partial Q}{\partial t}(\varphi,t)=-\frac{\partial P}{\partial t}(\theta,t)\frac{\partial Q}{\partial \varphi}(\varphi,t)=-\frac{\partial P}{\partial t}(\theta,t)\frac{1}{\rho\big(Q(\varphi,t),t\big)}\,,
\end{equation}
given (\ref{der_rho_comp}). Furthermore,
\begin{equation}
\nonumber
\frac{\partial P}{\partial t}(\theta,t)=\int_0^\theta \frac{\partial \rho}{\partial t} (s,t) ds=-\int_0^\theta \frac{\partial J}{\partial \theta}(s,t) ds =J_0(t)-J(\theta,t)
\end{equation}
and the expression (\ref{equa_dot_V}) becomes (from this point on, we drop the time variable)
\begin{equation}
\label{V_dot_bis}
\dot{\mathcal{V}}(\rho)=\int_0^1 \textrm{sign} \big(G(\varphi)\big) \frac{\partial}{\partial \varphi}\left(\frac{J\big(Q(\varphi)\big)-J_0}{\rho\big(Q(\varphi)\big)}\right) d\varphi\,.
\end{equation}

Assume that the function $G(\varphi)$ has a finite number of zero crossings, that is, a finite number $N_c$ of values $\varphi_c^{(0)}=0<\varphi_c^{(1)}<\cdots<\varphi_c^{(N_c)}<\varphi_c^{(N_c+1)}=1$
satisfying $G(\varphi_c^{(k)})=0$ with either a right or a left nonzero derivative. The assumption on $N_c$ will be relaxed at the end of the proof. Without loss of generality, assume that $G$ is nonnegative on even intervals and nonpositive on odd intervals, that is,
\begin{equation}
\label{ordre}
\displaystyle
(-1)^k\left(\frac{\partial Q}{\partial \varphi}-\frac{dQ^*}{d \varphi}\right)\geq 0 \qquad \forall \varphi\in[\varphi_c^{(k-1)},\varphi_c^{(k)}]\,.
\end{equation}

Next, computing the integral in \eqref{V_dot_bis}, we obtain
\begin{equation}
\nonumber
\dot{\mathcal{V}}(\rho)=2\sum_{k=1}^{N_c} (-1)^k \frac{J\big(Q(\varphi_c^{(k)})\big)-J_0}{\rho\big(Q(\varphi_c^{(k)})\big)}\,.
\end{equation}
Since it follows from (\ref{der_rho_comp}) and the equality $G(\varphi_c^{(k)})=0$ that
\begin{equation}
\label{pt_crit}
\rho\big(Q(\varphi_c^{(k)},t),t\big)=\rho^*\big(Q^*(\varphi_c^{(k)})\big)\,,\quad k=1,\dots,N_c\,,
\end{equation}
one obtains, given (\ref{rho_J}), (\ref{dynamics}), and (\ref{stat_density}),
\begin{eqnarray}
\dot{\mathcal{V}}(\rho)&=&2\sum_{k=1}^{N_c} (-1)^k \left[v\big(Q(\varphi_c^{(k)})\big)-\frac{J_0}{\rho^*\big(Q^*(\varphi_c^{(k)})\big)}\right] \nonumber \\
&=&2\sum_{k=1}^{N_c} (-1)^k \left[\omega+K\, Z\big(Q(\varphi_c^{(k)})\big) J_0-\frac{\omega J_0}{J^*}-K\, Z\big(Q^*(\varphi_c^{(k)})\big)J_0\right]\,. \label{V_inter}
\end{eqnarray}

The boundary condition (\ref{CL2}) yields a monotone relationship between the values $\rho(0)$ and $\rho(2\pi)$, that is, $\rho(0)>\rho^*(0)$ if and only if $\rho(2\pi)>\rho^*(2\pi)$. Apart from the case $\rho(0)=\rho^*(0)$, $\rho(2\pi)=\rho^*(2\pi)$, the number $N_c$ of values $\varphi^{(c)}$ satisfying (\ref{pt_crit}), is even, owing to the continuity of $\rho$ and $Q$ ($\rho^*$ and $Q^*$). Consequently, the terms $(-1)^k(\omega-\omega J_0/J^*)$ in (\ref{V_inter}) cancel each other. In the particular case $\rho(0)=\rho^*(0)$, $N_c$ is not necessarily even but it follows from (\ref{CL2}) that $J_0=J^*(0)=J^*$ and the above-mentioned terms are equal to zero. As a consequence, one obtains
\begin{eqnarray}
\dot{\mathcal{V}}(\rho) & = & 2\,J_0\,\sum_{k=1}^{N_c} (-1)^k K \left[Z\big(Q(\varphi_c^{(k)})\big)-Z\big(Q^*(\varphi_c^{(k)})\big)\right] \nonumber \\
& = & 2\,J_0\,\sum_{k=1}^{N_c} (-1)^k K\, Z'(\xi_k) \left[Q(\varphi_c^{(k)})-Q^*(\varphi_c^{(k)})\right] \label{sum} \\
& \triangleq & 2J_0\,\sum_{k=1}^{N_c} T^{(k)}\,, \nonumber
\end{eqnarray}
where the second equality is obtained through the mean value theorem, with \mbox{$\xi_k\in[Q(\varphi_c^{(k)}),Q^*(\varphi_c^{(k)})]$} or \mbox{$\xi_k\in[Q^*(\varphi_c^{(k)}),Q(\varphi_c^{(k)})]$}.

It remains to consider separately each term $T^{(k)}$ in the sum (\ref{sum}). For the sake of simplicity, we first consider the case $K\,Z''\geq 0$. Denoting \mbox{$(-1)^k [Q(\varphi_c^{(k)})-Q^*(\varphi_c^{(k)})]$} by $\Delta^{(k)} Q$, we also distinguish two cases: $\Delta^{(k)} Q>0$ and $ \Delta^{(k)} Q\leq 0$.

\textbf{Case $\mathbf{\Delta^{(k)} Q>0}$. }One has
\begin{equation}
\label{terme1_max}
T^{(k)}=K\, Z'(\xi_k) \Delta^{(k)} Q \leq \max_{\theta\in[0,2\pi]} \Big(K\,Z'(\theta)\Big) \Delta^{(k)} Q
\end{equation}
and
\begin{equation}
\label{terme1_min}
T^{(k)}=K\, Z'(\xi_k) \Delta^{(k)} Q \geq \min_{\theta\in[0,2\pi]} \Big(K\,Z'(\theta)\Big) \Delta^{(k)} Q\,.
\end{equation}

\textbf{Case $\mathbf{\Delta^{(k)} Q \leq 0}$. }We need to consider the addition of the term $T^{(k)}$ with the term $T^{(k-1)}$ or $T^{(k+1)}$. By (\ref{ordre}), one gets
\begin{eqnarray}
\label{equa_Q1}
-\Delta^{(k-1)} Q \leq \Delta^{(k)} Q & \leq & 0\,,\\
\label{equa_Q1b}
-\Delta^{(k+1)} Q \leq \Delta^{(k)} Q & \leq & 0\,.
\end{eqnarray}
It follows that
\begin{equation}
\label{eq_pos}
T^{(k-1)}+T^{(k)}=K\, Z'(\xi_{k-1})\Delta^{(k-1)} Q+ K\, Z'(\xi_k)\Delta^{(k)} Q \leq K\, Z'(\xi_k) \left(\Delta^{(k-1)} Q+\Delta^{(k)} Q\right)\,.
\end{equation}
The assumption $K\,Z''\geq 0$ implies \mbox{$K\, Z'(\xi_{k-1})\leq K\, Z'(\xi_k)$}, with \mbox{$\xi_{k-1}\leq \xi_k$} and the above inequality then follows from (\ref{equa_Q1}). In addition, (\ref{equa_Q1}) also implies that the right hand in inequality (\ref{eq_pos}) is the multiplication of $K\,Z'$ with the positive quantity $\Delta^{(k-1)} Q+\Delta^{(k)} Q$. Hence, (\ref{eq_pos}) can be rewritten as
\begin{equation}
\label{terme2_max}
T^{(k-1)}+T^{(k)}\leq \max_{\theta\in[0,2\pi]} \Big(K\,Z'(\theta)\Big)\left(\Delta^{(k-1)} Q+\Delta^{(k)} Q\right) \,.
\end{equation}
Similarly, considering the addition of the terms $T^{(k)}$ and $T^{(k+1)}$ and using (\ref{equa_Q1b}), one obtains
\begin{equation}
\label{terme2_min}
T^{(k)}+T^{(k+1)}\geq \min_{\theta\in[0,2\pi]} \Big(K\,Z'(\theta)\Big)\left(\Delta^{(k)} Q+\Delta^{(k+1)} Q\right) \,.
\end{equation}
\smallskip

Next, the inequalities (\ref{terme1_max}) and (\ref{terme2_max}) imply that the expression (\ref{sum}) can be rewritten as
\begin{eqnarray}
\dot{\mathcal{V}}(\rho)=2 J_0 \sum_{k=1}^{N_c} T^{(k)} && \leq 2 J_0\max_{\theta\in[0,2\pi]} \Big(K\,Z'(\theta)\Big) \sum_{k=1}^{N_c} \Delta^{(k)} Q \nonumber \\
&& = J_0 
\max_{\theta\in[0,2\pi]} \Big(K\,Z'(\theta)\Big) \int_0^1 \textrm{sign} \left(\frac{\partial Q}{\partial \varphi}-\frac{dQ^*}{d\varphi}\right) \left(\frac{\partial Q}{\partial \varphi}-\frac{dQ^*}{d\varphi}\right) d\varphi \nonumber \\
&& = J_0\max_{\theta\in[0,2\pi]} \Big(K\,Z'(\theta)\Big) \mathcal{V}(\rho) \label{inequa1} \,.
\end{eqnarray}

In the case $\Delta^{(k)} Q \leq 0$, the two terms $T^{(k)}$ and $T^{(k-1)}$ are considered together. The additional $T^{(k-1)}$ itself corresponds to the case $\Delta^{(k-1)} Q>0$, and does not require to be associated in turn with another term. In addition, there is no boundary problem since the term $T^{(1)}$ satisfies $\Delta^{(1)} Q>0$, given (\ref{ordre}) and $Q(0)=Q^*(0)=0$. 

Similarly, the inequalities (\ref{terme1_min}) and (\ref{terme2_min}) lead to
\begin{equation}
\label{inequa2}
\dot{\mathcal{V}}(\rho)=2 J_0 \sum_{k=1}^{N_c} T^{(k)} \geq J_0 \min_{\theta\in[0,2\pi]} \Big(K\,Z'(\theta)\Big) \mathcal{V}(\rho)\,.
\end{equation}
\smallskip

In the case $K\,Z''\leq 0$, the inequalities (\ref{terme2_max}) and (\ref{terme2_min}) are reversed, that is, the sum $T^{(k-1)}+T^{(k)}$ has a lower bound and the sum $T^{(k)}+T^{(k+1)}$ has an upper bound. Hence, the inequalities (\ref{inequa1}) and (\ref{inequa2}) still hold.

We have completed the proof assuming a finite number $N_c$ of zero crossings $G(\varphi_c^{(k)})=0$, but the reader will notice that the bounds (\ref{inequa1}) and (\ref{inequa2}) do not depend on $N_c$. For an arbitrary $\rho$, we conclude the argument as follows: by continuity, there exists a sequence of $\rho_i \rightarrow \rho$ such that every member of the corresponding sequence $G_i(\varphi) \rightarrow G(\varphi)$ has a finite number of zero crossings. Because the bounds (\ref{inequa1}) and (\ref{inequa2}) hold independently of $i$, they also hold in the limit, which concludes the proof.
\end{proof}

If there is no coupling or if the PRC is constant, \eqref{equa_fin} is a (marginally stable) standard transport equation and Theorem \ref{theo_Lyap_evol} implies that the Lyapunov function is constant along the solutions of \eqref{equa_fin}. This is in agreement with the fact that the total variation distance is a conserved quantity for most of the systems of conservation laws \cite{Serre}. But whereas the distance is constant with a standard transport equation, the distance is not constant under the influence of the coupling: the nonlinear coupling term in \eqref{equa_fin}, which depends on $K\,Z'$, induces a variation \eqref{mono_V} of the Lyapunov function. A monotone decreasing function $K\,Z$ implies a decreasing Lyapunov function along the solutions. This is discussed in detail in the next section.\\

The above result emphasizes the importance of considering (i) \emph{quantile densities instead of densities} and (ii) \emph{a $L^1$ distance (total variation distance) instead of a $L^2$ distance}. The importance of these two points is illustrated in the two following paragraphs.

\textbf{Quantile density vs. density function. }The Lyapunov function (\ref{Lyap_tot}) is induced by the total variation distance between \emph{quantile densities}. An alternative choice would be the total variation distance between \emph{density functions}, that is,
\begin{equation}
\label{Lyap_essai}
\mathcal{V}_\textrm{bis}(\rho)=\int_0^{2\pi}\left|\rho-\rho^* \right| d\theta \quad \forall \rho \in C^0([0,2\pi]\times \mathbb{R}^+;\mathbb{R}^+_0)\,.
\end{equation}
The time derivative is given by 
\begin{equation}
\label{deriv_Lyap_essai}
\dot{\mathcal{V}}_\textrm{bis}=\int_0^{2\pi} \textrm{sign} (\rho-\rho^*) \frac{\partial\rho}{\partial t} d\theta=\sum_{k=1}^{N_c+1} (-1)^k \left[J(\theta_c^{(k-1)})-J(\theta_c^{(k)})\right]\,,
\end{equation}
where we used (\ref{cont_equa}) and introduced the notation analog to (\ref{ordre})-(\ref{pt_crit}), that is, the values $\theta_c^{(k)}$ ($k=1,\dots,N_c$) satisfying
\begin{equation}
\label{equal_rho}
\rho(\theta_c^{(k)},t)=\rho^*(\theta_c^{(k)},t)\,.
\end{equation}
A simple argument shows that the Lyapunov function cannot be strictly decreasing along the solutions of (\ref{equa_fin}). Indeed, for any density satisfying $\rho(0)=\rho^*(0)$, one must also have $J_0=J^*(0)=J^*$ since it follows from (\ref{CL2}) that there is a bijection between the values $\rho(0)$ and $J_0$. But then, (\ref{rho_J}) and (\ref{equal_rho}) imply that \mbox{$J(\theta_c^{(k)})=J^*(\theta_c^{(k)})=J^*$} for all $k$ and the derivative (\ref{deriv_Lyap_essai}) leads to $\dot{\mathcal{V}}_\textrm{bis}=0$. Even though one actually shows that $\dot{\mathcal{V}}\leq 0$, La Salle principle cannot be used to prove the stability, since it is not obvious to prove the precompactness of the trajectories.

This argument shows that a direct application of the total variation distance on density functions does not lead to a good candidate Lyapunov function for the PDE (\ref{equa_fin}). A key point is to apply the total variation distance on \emph{quantile functions} instead.

\textbf{$\mathbf{L^1}$ distance vs. $\mathbf{L^2}$ distance. }The Lyapunov function (\ref{Lyap_tot}) is induced by a \emph{$L^1$ distance}. An alternative choice would be a Lyapunov function induced by a (more common) \emph{$L^2$ distance}, that is $\mathcal{V}_\textrm{ter}=\|q-q^*\|_{L^2}$. However, straightforward computations (not presented here) show that this candidate Lyapunov function satisfies $\dot{\mathcal{V}}_\textrm{ter}(\rho)>0$ for some $\rho$. This remark also applies to the finite dimensional case.

\section{Convergence analysis for monotone oscillators}
\label{well_posedness}

The result of Theorem \ref{theo_Lyap_evol} has a strong implication in the case of monotone PRC's. It implies that the Lyapunov function (\ref{Lyap_tot}) has a monotone time evolution if the PRC is monotone. In this situation, the Lyapunov function either converges to a lower bound or to an upper bound. These two bounds correspond to the two particular behaviors which characterize the dichotomy highlighted in Section \ref{dicho}. They are given by
\begin{equation}
\nonumber
0\leq\mathcal{V}(\rho)\leq \left\|q\right\|_{L^1}+\left\|q^*\right\|_{L^1}=Q(1)-Q(0)+Q^*(1)-Q^*(0)=4\pi\,.
\end{equation}
At the lower bound, the Lyapunov function is equal to zero if and only if the density corresponds to the asynchronous (stationary) density (\ref{stat_density}). On the other hand, we will show that the function tends to the upper bound $4\pi$ if the density $\rho$ tends to a Dirac function (synchronization).

\subsection{Exponential convergence to the asynchronous state}
\label{conv_as}

Theorem \ref{theo_Lyap_evol} will be used to study convergence to the asynchronous state for monotone decreasing functions $K\, Z(\theta)$. However, in order to apply (\ref{mono_V}) along the solutions, we need to show independently that the flux $J_0(t)$ remains strictly positive and uniformly bounded for all time. This condition will restrict the set of admissible initial conditions. In particular, the initial conditions have to ensure that $0<J_0<\infty$ when any oscillator crosses $\theta=2\pi$ for the first time. Formally, we consider the characteristic curves $\Lambda_{\theta}(t)$ defined by $\dot{\Lambda}_{\theta}=v(\Lambda_{\theta}(t),t)$, $\Lambda_{\theta}(0)=\theta\in[0,2\pi]$, $\Lambda_{\theta}(\overline{t}_{\theta})=2\pi$. 

The (strictly positive) initial density $\rho(\theta,0)>0$ must be such that the value of the flux at the intersection of the characteristic curves with $\theta=2\pi$ is strictly positive and bounded, that is
\begin{equation}
\label{IC}
\rho(\theta,0)>0 \quad \Rightarrow \quad
0<J\big(\Lambda_\theta(\overline{t}_\theta)=2\pi,\overline{t}_{\theta}\big)=J_0(\overline{t}_\theta)<\infty \quad \forall\,\theta\in[0,2\pi]\,.
\end{equation}
Condition (\ref{IC}) is the condition that is imposed on the initial conditions to ensure that the flux satisfies $0<J_0(t)<\infty$ for $t\in[0,\overline{t}_{\theta=0}]$. But if the flux $J_0(t)$ is bounded for $t\in[0,\overline{t}_{\theta=0}]$, then the flux $J_0(t)$ is uniformly bounded for all time and Theorem \ref{theo_Lyap_evol} can be applied to prove the exponential decreasing of the Lyapunov function. The result is summarized in the following proposition.
\begin{proposition}
\label{prop_dec}
Consider the transport PDE (\ref{equa_fin})-(\ref{CL2}) and assume that $Z(\theta)$ is such that (i) the stationary density (\ref{stat_density}) exists, (ii) $K\,Z'(\theta)<0$ $\forall \theta\in[0,2\pi]$, and (iii) either $Z''(\theta)\geq0$ $\forall \theta\in[0,2\pi]$ or $Z''(\theta)\leq0$ $\forall \theta\in[0,2\pi]$. Then all solutions \mbox{$\rho(\theta,t)\in C^0([0,2\pi]\times \mathbb{R}^+;\mathbb{R}^+)$}, with $\|\rho\|_{L_1}=1$ and with an initial condition satisfying (\ref{IC}), exponentially converge to the asynchronous state and the Lyapunov function (\ref{Lyap_tot}) is exponentially decreasing along them.
\end{proposition}
\begin{proof}
We proceed in two steps.\\
\textbf{Step 1. } Let $[\underline{t}_1=0,\overline{t}_1]$, \dots, $[\underline{t}_i,\overline{t}_i]$, with $\underline{t}_{i}=\overline{t}_{i-1}$, be the successive time intervals so that the characteristic curves $\Lambda_{[\underline{t}_i,\overline{t}_i]}$ satisfy $\Lambda_{[\underline{t}_i,\overline{t}_i]}(\underline{t}_i)=0$ and $\Lambda_{[\underline{t}_i,\overline{t}_i]}(\overline{t}_i)=2\pi$. We will show that $J_0(t)$ is uniformly bounded by
\begin{equation}
\label{uniform_bound}
J^{\textrm{min}}\leq J_0(t)\leq J^{\textrm{max}}\qquad \forall t>0\,,
\end{equation}
with the bounds $J^{\textrm{min}}=\min_{t\in[0,\overline{t}_1]} J_0(t)$ and $J^{\textrm{max}}<\max_{t\in[0,\overline{t}_1]} J_0(t)$. 

We consider a characteristic curve $\Lambda_{[\underline{t},\overline{t}]}(t)$ with $[\underline{t},\overline{t}]\subset[\underline{t}_{i-1},\overline{t}_i]$ (i.e. $\underline{t}\in[\underline{t}_{i-1},\overline{t}_{i-1}]$ and $\overline{t}\in[\underline{t}_i,\overline{t}_i]$). Solving the total derivative equation (\ref{tot_der}) on $\Lambda_{[\underline{t},\overline{t}]}(t)$ yields
\begin{equation}
\label{equa_expo_evol_tho}
\rho(0,\underline{t})\exp \left(-\int_{\underline{t}}^{\overline{t}} J_0(t)\, K\,Z'\big(\Lambda_{[\underline{t},\overline{t}]}(t)\big) \, dt\right)=\rho(2\pi,\overline{t})\,.
\end{equation}
Next, using (\ref{rho_J}) and expressing the integral in the space variable $\Lambda_{[\underline{t},\overline{t}]}$ along the characteristic curve lead to
\begin{equation}
\label{eq_integral_comp}
\frac{J_0(\underline{t})}{\omega+K\,Z(0)\,J_0(\underline{t})} \exp \left(-\int_{0}^{2\pi} \frac{J_0}{\omega+K\, Z(\Lambda_{[\underline{t},\overline{t}]})\, J_0} K \,Z'(\Lambda_{[\underline{t},\overline{t}]}) \, d\Lambda_{[\underline{t},\overline{t}]}\right)=\frac{J_0(\overline{t})}{\omega+K\, Z(2\pi)\, J_0(\overline{t})}\,.
\end{equation}
Since $\rho(\theta,0)>0$, the exponential evolution \eqref{equa_expo_evol_tho} of the density along (all) the characteristic curves implies that the density remains strictly positive for all $t>0$. Hence, the flux is also strictly positive for all $t>0$. Then, we define \mbox{$J^{\textrm{max}}_{[\underline{t},\overline{t}]} \triangleq \max_{t\in[\underline{t},\overline{t}]} J_0(t)>0$} and we can define
\begin{equation}
\label{equa_C_M}
\begin{split}
C_{i} \triangleq & \max_{[\underline{t},\overline{t}]\subset[\underline{t}_{i-1},\overline{t}_i]}  \Bigg\{\int_{0}^{2\pi} \frac{J_0}{\omega+K\, Z(\Lambda_{[\underline{t},\overline{t}]})\, J_0} K \,Z'(\Lambda_{[\underline{t},\overline{t}]}) \, d\Lambda_{[\underline{t},\overline{t}]}\\
& \qquad \qquad \Big/ \int_{0}^{2\pi} \frac{J^{\textrm{max}}_{[\underline{t},\overline{t}]}}{\omega+K\, Z(\Lambda_{[\underline{t},\overline{t}]})\, J^{\textrm{max}}_{[\underline{t},\overline{t}]}} K \,Z'(\Lambda_{[\underline{t},\overline{t}]}) \, d\Lambda_{[\underline{t},\overline{t}]} \Bigg\}\,.
\end{split}
\end{equation}
Since $K\, Z'<0$ $\forall \theta$, one has $C_i=1$ only if $J_0(t)$ is constant on at least one interval $[\underline{t},\overline{t}]$. In this case, the solution has reached the steady state and it follows that
\begin{equation}
\label{equa_J_max_first}
J_0(t)=J^* \qquad \forall t\in[\underline{t}_i,\infty)\,.
\end{equation}
Otherwise, one has $C_i<1$ and, using \eqref{equa_C_M} and computing the integral in equation \eqref{eq_integral_comp}, we obtain
\begin{equation}
\frac{J_0(\overline{t})}{\omega+K\,Z(2\pi)\, J_0(\overline{t})} \leq \frac{J_0(\underline{t})}{\omega+K\,Z(0)\,J_0(\underline{t})} \left(\frac{\omega+K\,Z(0)\,J^{\textrm{max}}_{[\underline{t},\overline{t}]}}{\omega+K\,Z(2 \pi)\,J^{\textrm{max}}_{[\underline{t},\overline{t}]}}\right)^{C_i}\,.
\end{equation}
Since $J_0(\underline{t})\leq J^{\textrm{max}}_{[\underline{t},\overline{t}]}$ by definition, it follows that
\begin{equation}
\nonumber
\frac{J_0(\overline{t})}{\omega+K\,Z(2\pi)\, J_0(\overline{t})} \leq  \frac{J^{\textrm{max}}_{[\underline{t},\overline{t}]}}{\omega+K\,Z(2\pi)\,J^{\textrm{max}}_{[\underline{t},\overline{t}]}}
\underbrace{\max_{[\underline{t},\overline{t}]\subset[\underline{t}_{i-1},\overline{t}_i]} \left(\frac{\omega+K\,Z(2\pi)\,J^{\textrm{max}}_{[\underline{t},\overline{t}]}}{\omega+K\,Z(0)\,J^{\textrm{max}}_{[\underline{t},\overline{t}]}}\right)^{1-C_i}}_{<1}\,,
\end{equation}
which implies that there exists a constant $C_i^\textrm{max}<1$ (see Remark \ref{rem_constant}) such that
\begin{equation}
\label{equa_ineg_flux}
J_0(\overline{t})\leq C_i^\textrm{max} J^{\textrm{max}}_{[\underline{t},\overline{t}]} \qquad \forall \overline{t}\in[\underline{t}_i,\overline{t}_i] \,.
\end{equation}
It follows from \eqref{equa_ineg_flux} that
\begin{equation}
\label{equa_J_max_2nd}
J_0(t)\leq C_i^\textrm{max} J^{\textrm{max}}_{[\underline{t}_{i-1},\overline{t}_{i-1}]} \qquad \forall t \in[\underline{t}_i,\overline{t}_i]
\end{equation}
and a straightforward induction argument using \eqref{equa_J_max_first} and \eqref{equa_J_max_2nd} implies that 
\begin{equation*}
J_0(t) \leq \left(\prod_{k=2}^i C_k^\textrm{max}\right) J^{\textrm{max}}_{[0,\overline{t}_1]} \leq J^{\textrm{max}} \quad \forall t>0\,,
\end{equation*}
where some constants $C_k^\textrm{max}$ are equal to one if \eqref{equa_J_max_first} is satisfied at time $t$. The proof for the lower bound $J^{\textrm{min}}$ in \eqref{uniform_bound} follows on similar lines.

\textbf{Step 2: } Since $\rho(\theta,t)>0$ for all $t>0$, the result of Theorem \ref{theo_Lyap_evol} is applied and yields
\begin{equation}
\nonumber
J_{\textrm{min}} \min_{\theta\in[0,2\pi]} \Big(K\,Z'(\theta)\Big) \mathcal{V}(\rho) \leq \dot{\mathcal{V}}(\rho) \leq J_{\textrm{max}} \max_{\theta\in[0,2\pi]} \Big(K\,Z'(\theta)\Big) \mathcal{V}(\rho)
\end{equation}
or, equivalently,
\begin{equation}
\nonumber
-J_{\textrm{max}} \max_{\theta\in[0,2\pi]} \Big|K\,Z'(\theta)\Big| \mathcal{V}(\rho) \leq \dot{\mathcal{V}}(\rho) \leq -J_{\textrm{min}} \min_{\theta\in[0,2\pi]}  \Big|K\,Z'(\theta)\Big| \mathcal{V}(\rho) \leq 0
\end{equation}
since $K\,Z'<0$. The initial condition (\ref{IC}) implies that $J^{\textrm{min}}>0$ and $J^{\textrm{max}}<\infty$ and the Lyapunov function is exponentially decreasing, which concludes the proof.

\begin{remark}\label{rem_constant} When $KZ(2\pi)<0$, in the particular case $J^{\textrm{max}}_{[\underline{t},\overline{t}]}=\omega/|KZ(2\pi)|$, \eqref{equa_ineg_flux} does not hold. However, since $\rho(\theta,t)>0$ for all $t>0$, the value $\omega/|KZ(2\pi)|$ can be used as the upper bound $J^{\textrm{max}}$ in \eqref{uniform_bound}.
\end{remark}

\end{proof}

Proposition \ref{prop_dec} is a strong result showing that, provided that the function $K\, Z(\theta)$ is decreasing, the solution $\rho(\cdot,t)$ remains, for all time, in a particular set of functions \mbox{$\{\rho|\mathcal{V}(\rho)<C\}$}, with the constant $C>0$. Inside this set, the solution eventually converges at exponential rate toward the stationary solution $\rho^*$, corresponding to $\mathcal{V}(\rho^*)=0$.\\

The restriction \eqref{IC} on the initial condition is rather weak for decreasing functions $K\,Z(\theta)$, as shown by the following proposition.
\begin{proposition}
\label{prop_CI_dec}
Assume that $K\, Z'(\theta)<0$ $\forall \theta\in[0,2\pi]$. If $K\,Z(2\pi)\leq 0$, then (\ref{IC}) is always satisfied. If $K\,Z(2\pi)>0$, then (\ref{IC}) is satisfied if
\begin{equation}
\label{cond_CI2b}
\rho(\theta,0)<\frac{1}{K\,Z(\theta)} \quad \forall \theta \in [0,2\pi]\,.
\end{equation}
\end{proposition}
\begin{proof}
\textbf{Case $\mathbf{K\,Z(2\pi)\leq 0}$.} Using the boundary condition (\ref{CL2}), the condition $0<J_0(\overline{t}_\theta)<\infty$ is turned into a condition on $\rho(2\pi,\overline{t}_\theta)$ and (\ref{IC}) is equivalent to
\begin{equation}
\label{cond_CI1}
\rho(\theta,0)>0 \quad \Rightarrow \quad
\rho(2\pi,\overline{t}_\theta)>0 \quad \forall\,\theta\in[0,2\pi]\,.
\end{equation}
For $\theta=2\pi$, condition (\ref{cond_CI1}) is satisfied, since $\overline{t}_{2\pi}=0$. Next, we proceed by induction on $\theta$: given $\theta$ and assuming that (\ref{cond_CI1}) is satisfied for all $\tilde{\theta}>\theta$, we prove that (\ref{cond_CI1}) also holds at $\theta$. The total derivative equation (\ref{tot_der}) is well-defined on the characteristic curve $\Lambda_{\theta}(t)$ since $0<J_0(t)<\infty$ for all $t<\overline{t}_{\theta}$. Solving (\ref{tot_der}) along the characteristic curve yields
\begin{equation}
\label{charac_sol}
\rho(2\pi,\overline{t}_{\theta})=\rho(\theta,0) \, \exp\left(-\int_0^{\overline{t}_{\theta}} J_0(t) \, K \,Z'\big(\Lambda_{\theta}(t)\big)\, dt\right)
\end{equation}
and $\rho(\theta,0)>0$ implies \mbox{$\rho(2\pi,\overline{t}_{\theta})>0$}. Condition (\ref{IC}), equivalent to (\ref{cond_CI1}), is then always satisfied.

\textbf{Case $\mathbf{K\,Z(2\pi)>0}$.} Using (\ref{CL2}), the condition \mbox{$0<J_0(\overline{t}_\theta)<\infty$} is turned into a condition on $\rho(2\pi,\overline{t}_\theta)$ and (\ref{IC}) is equivalent to
\begin{equation}
\label{cond_CI2}
\rho(\theta,0)>0 \quad \Rightarrow \quad
0<\rho(2\pi,\overline{t}_\theta)<\frac{1}{K\,Z(2\pi)} \quad \forall\,\theta\in[0,2\pi]\,.
\end{equation}
The strict condition $\rho(\theta,0)>0$ always implies \mbox{$\rho(2\pi,\overline{t}_\theta)>0$}, as in the case \mbox{$K\,Z(2\pi)\leq 0$}. Hence, we focus on the additional upper bound on the density $\rho(2\pi,\overline{t}_\theta)$. For $\theta=2\pi$, condition (\ref{cond_CI2b}) implies (\ref{cond_CI2}), since $\overline{t}_{2\pi}=0$. Next, we proceed by induction on $\theta$: given $\theta$ and assuming that (\ref{cond_CI2}) is satisfied for all $\tilde{\theta}>\theta$, we prove that (\ref{cond_CI2}) also holds at $\theta$ (provided that condition (\ref{cond_CI2b}) is satisfied). Using (\ref{charac_sol}) with condition (\ref{cond_CI2b}) leads to
\begin{equation}
\nonumber
\rho(2\pi,\overline{t}_{\theta})<\frac{1}{K\,Z(\theta)}\exp\left(-\int_0^{\overline{t}_{{\theta}}} J_0(t) \, K \,Z'\big(\Lambda_{{\theta}}(t)\big)\, dt\right)\,.
\end{equation}
Expressing the integral in the space variable along the characteristic curve yields
\begin{equation}
\nonumber
\rho(2\pi,\overline{t}_{\theta})<\frac{1}{K\,Z(\theta)} \exp\left(-\int_{\theta}^{2\pi} \frac{J_0}{\omega+K\,Z(\Lambda_\theta)\, J_0} K\,Z'(\Lambda_\theta) \, d\Lambda_\theta \right)\,.
\end{equation}
Since $K\,Z'$ is negative, the flux $J_0$ can be replaced by its maximal value, that is, $J_0\rightarrow \infty$ and the above inequality leads to
\begin{equation}
\label{inequa_inter}
\rho(2\pi,\overline{t}_{\theta})<\frac{1}{K\,Z(\theta)} \exp\left(-\int_{\theta}^{2\pi} \frac{1}{Z(\Lambda_\theta)} \,Z'(\Lambda_\theta) \, d\Lambda_\theta \right)\,.
\end{equation}
The relation (\ref{inequa_inter}) is well-defined since $Z(\theta)\neq 0$ for all $\theta\in[0,2\pi]$. Finally, computing the integral in (\ref{inequa_inter}) implies that $\rho(2\pi,\overline{t}_{\theta})<1/[K\,Z(2\pi)]$. Condition (\ref{IC}), equivalent to (\ref{cond_CI2}), is then satisfied. This concludes the proof.
\qquad \end{proof}

When the conditions of Proposition \ref{prop_CI_dec} are not satisfied, condition (\ref{IC}) may fail to hold, in which case the flux blows-up (finite escape time to infinity). In the case $K\,Z(2\pi)>0$, there is a positive feedback between the flux and the velocity: a high value of $J_0$ increases the velocity through the coupling, which in turn increases the flux. When the density approaches the critical value
\begin{equation}
\label{cond_absorp}
\rho(2\pi,t)=\frac{1}{K\, Z(2\pi)}\,,
\end{equation}
the flux is high enough to blow-up through the positive feedback.

This finite escape time phenomenon is related to the \textsl{absorption} phenomenon observed for finite populations. If two oscillators are close enough, the firing of one oscillator can trigger the instantaneous firing of the second: the latter is \textsl{absorbed} by the former. In particular, condition \eqref{cond_CI2b} is equivalent for finite populations to an initial condition imposing a minimal state distance $K/N$ between any two oscillators, a condition that prevents the absorption phenomenon.

\subsection{Finite time convergence to a synchronous state} 
\label{subsec_synchronous}

For increasing functions $K\,Z(\theta)$, Theorem \ref{theo_Lyap_evol} implies that the Lyapunov function (\ref{Lyap_tot}) is strictly increasing, and a synchronous behavior is observed in finite time, for any initial condition. Either the flux $J_0(t)$ becomes infinite in finite time or the density $\rho(0,t)$ becomes infinite in finite time.

The finite time convergence to synchronization is established in Proposition \ref{prop_sync}. As a preliminary to this result, we need the following lemma.

\begin{lemma}
\label{lemma_TVD}
The Lyapunov function (\ref{Lyap_tot}) satisfies
\begin{equation}
\nonumber
\mathcal{V}=\|q-q^*\|_{L^1}\leq 4\pi-2\, q_{\textrm{min}}
\end{equation}
with
\begin{equation}
\nonumber
q_{\textrm{min}}=\min\left(\min_{\varphi\in[0,1]} q(\varphi),\min_{\varphi\in[0,1]} q^*(\varphi)\right)\,.
\end{equation}
\end{lemma}
\begin{proof}
The proof of Lemma \ref{lemma_TVD} can be found in Appendix.
\end{proof}

Through Theorem \ref{theo_Lyap_evol} and the preceding lemma, the following proposition establishes the finite time convergence to the synchronous state.
\begin{proposition}
\label{prop_sync}
Consider the transport PDE (\ref{equa_fin})-(\ref{CL2}) and assume that $Z(\theta)$ is such that (i) the stationary density (\ref{stat_density}) exists, (ii) $K\,Z'(\theta)>0$ $\forall \theta\in[0,2\pi]$, and (iii) either $Z''(\theta)\geq0$ $\forall \theta\in[0,2\pi]$ or $Z''(\theta)\leq0$ $\forall \theta\in[0,2\pi]$. Then all solutions \mbox{$\rho(\theta,t)\in C^0([0,2\pi]\times [0,t_{\textrm{fin}});\mathbb{R}^+)$}, $t_{\textrm{fin}}<\infty$, with $\|\rho\|_{L_1}=1$ and with an initial condition $\rho(\theta,0)>0$, converge in finite time to a synchronous state. That is, if $K\,Z(0)\geq 0$, the flux satisfies $J(0,t_{\textrm{fin}})=\infty$, or if $K\,Z(0)<0$, the density satisfies $\rho(0,t_{\textrm{fin}})=\infty$.
\end{proposition}
\begin{proof}
An infinite flux $J_0(t)$ or an infinite density $\rho(0,t)$ is obtained when the density $\rho(2\pi,t)$ reaches a critical value. If $K\,Z(0)\geq 0$, the value $K\,Z(2\pi)$ is positive since $K\,Z$ is increasing. Hence, the flux $J_0$ becomes infinite when the density $\rho(2\pi,t)$ exceeds the critical value (\ref{cond_absorp}). If $K\,Z(0)<0$, the velocity (\ref{dynamics}) at $\theta=0$ is equal to zero when the flux reaches the value $J_0(t)=\omega/|K\,Z(0)|$ or equivalently, given (\ref{CL2}), when the density reaches the value
\begin{equation}
\label{cond_vel_zero}
\rho(2\pi,t)=\frac{1}{K\,Z(2\pi)-K\,Z(0)}\,.
\end{equation}
With a velocity equal to zero at $\theta=0$, the relationship (\ref{rho_J}) implies that the density is infinite at $\theta=0$. (If $K\,Z(0)<0$ along with $K\,Z(2\pi)>0$, the reader will notice that the value (\ref{cond_absorp}) has no importance, since (\ref{cond_absorp}) is greater than (\ref{cond_vel_zero}).) 

Next, we show that the density $\rho(2\pi,t)$ must necessarily reach the critical value (\ref{cond_absorp}) or (\ref{cond_vel_zero}) in finite time $t_{\textrm{fin}}$. Let us consider a characteristic curve $\Lambda(t)$, with $\Lambda(\underline{t})=0$ and $\Lambda(\overline{t})=2\pi$ and assume that the synchronous state is not reached within $[\underline{t},\overline{t}]$, so that $\rho(\theta,t)\in C^0([0,2\pi]\times[\underline{t},\overline{t}];\mathbb{R}^+_0)$. (The initial condition $\rho(\theta,0)>0$ implies that the density remains strictly positive on (all) the characteristic curves.) Applying Theorem \ref{theo_Lyap_evol} and integrating (\ref{mono_V}), one has
\begin{equation}
\nonumber
\min_{\theta\in[0,2\pi]} \Big(K\,Z'(\theta)\Big) \int_{\underline{t}}^{\overline{t}} J_0(t)\,dt \leq \int_{\mathcal{V}(\underline{t})}^{\mathcal{V}(\overline{t})} \frac{1}{\mathcal{V}} \, d\mathcal{V} \leq  \max_{\theta\in[0,2\pi]} \Big(K\,Z'(\theta)\Big) \int_{\underline{t}}^{\overline{t}} J_0(t) \, dt \,.
\end{equation}
Since $[\underline{t},\overline{t}]$ is the time interval corresponding to the complete evolution of an oscillator from $\theta=0$ to $\theta=2\pi$ , the integral of the flux $J_0$ in the above equation is equal to one and it follows that
\begin{equation}
\nonumber
\exp \left(\min_{\theta\in[0,2\pi]} \Big(K\,Z'(\theta)\Big)\right) \leq \frac{\mathcal{V}(\overline{t})}{\mathcal{V}(\underline{t})}  \leq\exp \left(\max_{\theta\in[0,2\pi]} \Big(K\,Z'(\theta)\Big)\right) \,.
\end{equation}
The condition $K\,Z'>0$ implies that the Lyapunov function strictly increases within the time interval $[\underline{t},\overline{t}]$. Considering $n$ successive intervals $[0,\overline{t}_1]$, \dots, $[\underline{t}_n,\overline{t}_n]$, with $\underline{t}_{i+1}=\overline{t}_i$, one obtains
\begin{equation}
\nonumber
\frac{\mathcal{V}(\overline{t}_n)}{\mathcal{V}(0)} \geq \exp \left(n\,\min_{\theta\in[0,2\pi]} \Big(K\,Z'(\theta)\Big)\right)\,.
\end{equation}
Hence, a given value $\overline{\mathcal{V}}>\mathcal{V}(0)$ is reached within at most $n_{\textrm{max}}$ time intervals, with
\begin{equation}
\nonumber
n_{\textrm{max}}\leq \frac{\log\big({\overline{\mathcal{V}}}/\mathcal{V}(0)\big)}{\min_{\theta\in[0,2\pi]} \Big(K\,Z'(\theta)\Big)}\,.
\end{equation}
Since the time length of each interval $[\underline{t}_k,\overline{t}_k]$ ($k=1,\dots,n_{\textrm{max}}$) is finite, any value $\overline{\mathcal{V}}<4\pi$ is reached in finite time. By Lemma \ref{lemma_TVD}, $\mathcal{V}=\overline{\mathcal{V}}$ implies that $q_{\textrm{min}}\leq (4\pi-\overline{\mathcal{V}})/2$. When considering values $\overline{\mathcal{V}}$ close to $4\pi$, the minimum of the quantile density $q$ reaches in finite time any given value close to zero. Given (\ref{der_rho_comp}), this implies that the maximum of the density $\rho_M$ reaches in finite time any given value (provided that the critical value (\ref{cond_absorp}) or (\ref{cond_vel_zero}) is not already reached). In particular, the value
\begin{equation}
\label{rho_max}
\rho_M=\rho_c \exp \left(\max_{\theta\in[0,2\pi]} \Big(K\,Z'(\theta)\Big)\right)\,,
\end{equation}
with $\rho_c$ denoting the critical value (\ref{cond_absorp}) or (\ref{cond_vel_zero}), is reached in finite time. Then, the maximum value $\rho_M$ (obtained at $\theta_M$ at time $t_M$) decreases along the characteristic curve $\Lambda(t)$, with $\Lambda(t_M)=\theta_M$. Solving the total derivative equation (\ref{tot_der}), one shows that the variation along the characteristic curve is bounded:
\begin{equation}
\nonumber
\rho(2\pi,t_{\textrm{fin}})=\rho_M \exp \left(-\int_{t_M}^{t_{\textrm{fin}}} J_0(t) \, K\,Z'\big(\Lambda(t)\big) dt \right) \geq \rho_M \exp \left(-\max_{\theta\in[0,2\pi]} \Big(K\,Z'(\theta)\Big)\right)\,,
\end{equation}
where the inequality is obtained since the integral of $J_0$ on $[t_M,t_\textrm{fin}]$ is less than one. With $\rho_M$ given by (\ref{rho_max}), the density exceeds the critical value in finite time $t_{\textrm{fin}}$, which concludes the proof.
\qquad \end{proof}

Proposition \ref{prop_sync} shows that oscillators with monotone increasing functions $K\,Z$ converge to a synchronous behavior corresponding to a solution with a finite escape time to infinity. If the coupling is excitatory ($0\leq K\,Z(0)<K\,Z(2\pi)$), the flux $J_0$ blows-up and is infinite at time $t_{\textrm{fin}}$, as well as the velocity $v(2\pi,t_{\textrm{fin}})$. If the coupling is inhibitory ($K\,Z(0)<0$), the velocity $v(0,t)$ reaches the value zero at time $t_{\textrm{fin}}$ and the oscillators accumulate at $\theta=0$, yielding an infinite density $\rho(0,t_{\textrm{fin}})$.

\section{Implications for populations of oscillators}
\label{discussion}

The results of Section \ref{main_result} and Section \ref{well_posedness} are technical. In this section, their implications on the behavior of (monotone) pulse-coupled oscillators are discussed with more details.

Emphasis is put on the strong parallel between the analysis of the present paper on infinite populations and earlier results on finite populations. We also consider more general models and discuss the relevance of an instantaneous impulsive coupling.

\subsection{Parallel with finite populations}
\label{subsec_parallel}

The results of Proposition \ref{prop_dec} and Proposition \ref{prop_sync} apply to oscillators characterized by a monotone PRC. In the case of integrate-and-fire oscillators $\dot{x}=F(x)$, the results thereby apply to oscillators with a monotone vector field $F$ (including the popular LIF oscillators) and reinforce the numerical observations presented in Section \ref{subsec_asymp_simu}.

When $K\, Z'<0$ ($K\,dF/dx>0$), the oscillators converge to the asynchronous state.
\begin{theorem}
\label{theo_cont_1}
Consider a continuum of identical pulse-coupled integrate-and-fire oscillators $\dot{x}=F(x)$ characterized by (i) a monotone dynamics $K\,dF/dx>0$ and (ii) a PRC with a curvature of constant sign. Then, provided that the asynchronous state exits (cf. Corollary \ref{corol_stat_dens}), either the initial conditions do not fulfill \eqref{IC} and the flux tends to infinity in finite time, or the continuum exponentially converges to the asynchronous state.
\end{theorem}
When $K\, Z'>0$ ($K\,dF/dx<0$), the oscillators achieve synchrony.
\begin{theorem}
\label{theo_cont_2}
Consider a continuum of identical pulse-coupled integrate-and-fire oscillators $\dot{x}=F(x)$ characterized by (i) a monotone dynamics $K\,dF/dx<0$ and (ii) a PRC with a curvature of constant sign. Then, provided that the asynchronous state exits (cf. Corollary \ref{corol_stat_dens}), the continuum converges to the synchronous state in finite time.
\end{theorem}

The curvature condition on the PRC is verified for the LIF oscillators (with dynamics $\dot{x}=S-\gamma\,x$). Indeed, the PRC is given by
\begin{equation*}
Z(\theta)=\frac{\omega}{S}\exp\left(\frac{\gamma\,\theta}{\omega}\right) \quad \textrm{with} \quad \omega=2\pi\, \gamma \left[\log\left(\frac{S}{S-\gamma}\right)\right]^{-1}
\end{equation*}
and satisfies $Z''>0$. Hence, Theorem \ref{theo_cont_1} and Theorem \ref{theo_cont_2} prove the global dichotomic behavior of the continuum of pulse-coupled LIF oscillators: the oscillators converge to the synchronous state when the coupling is excitatory ($K>0$) and to the asynchronous state when the coupling is inhibitory ($K<0$). These global results complement the local results on the continuum model presented in \cite{Abbott, Kuramoto3, Vreeswijk}.

The results of Theorem \ref{theo_cont_1} and Theorem \ref{theo_cont_2}, for infinite populations, are the exact analogs of the results presented in our previous study \cite{Mauroy} and in \cite{Mirollo}, for finite populations. They prove that the behaviors of monotone oscillators are dichotomic, not only for finite populations but also for infinite populations, thereby drawing a strong parallel between the analysis of finite populations and the analysis of infinite populations.

\subsection{Application to other oscillators} 

Beyond the case of integrate-and-fire oscillators, the results developed in the present paper apply to more general phase dynamics. The PRC can be (numerically) computed from general high-dimensional state models possessing a stable limit cycle, in which case the phase dynamics \eqref{phase_model} represents a one-dimensional reduced model valid in the neighborhood of the limit cycle \cite{Smeal, Winfree}. However, two limitations appear, which require (i) a weak coupling strength and (ii) a monotone PRC.

(i) Phase reduction of multidimensional models (other than integrate-and-fire models) is valid as long as the coupling is weak. Since the derivation of the transport equation relies on the phase dynamics, it is relevant only for a coupling strength $K\ll 1$.

(ii) The results only apply to strictly monotone PRC's. As a consequence of the monotonicity, the PRC must be characterized by $Z(0)\neq Z(2\pi)$ and cannot be periodic. However, a PRC computed on a limit cycle is periodic, unless the limit cycle is discontinuous. For instance, such a discontinuity is artificially created in the one-dimensional integrate-and-fire model (or in the Izhikhevich model \cite{Izi_book,Izhikevich2}). A discontinuous limit cycle can also be obtained for models with separated time scales (e.g. relaxation models such as FitzHugh-Nagumo oscillators \cite{FitzHugh,Nagumo} or Van der Pol oscillators, spiking oscillators). Since the time spent on a part of the limit cycle is negligible with respect to the time spent on the rest of the cycle, the fast part can be replaced in good approximation by a discontinuity. The oscillators are thereby characterized in good approximation by a non periodic PRC, that can be monotone.

This situation is illustrated in the following example, for oscillators characterized by a limit cycle close to a homoclinic bifurcation. 
\begin{example}In \cite{Brown}, limit cycles close to a homoclinic bifurcation are approximated by a discontinuous limit cycle, to which corresponds a monotone (discontinuous) PRC that satisfies the additional curvature condition, so that the results apply.

A homoclinic bifurcation occurs when there exists, for a given parameter value, a homoclinic orbit to a saddle point with real eigenvalues. At the bifurcation, a limit cycle appears. Assuming that there is a single unstable eigenvalue $\lambda_u$ such that $\lambda_u<|\lambda_{s_j}|$, with $\lambda_{s_j}$ being the stable eigenvalues, the limit cycle is stable \cite{Guckenheimer}. Since the trajectory is much slower near the saddle point, the limit cycle is discontinuous in good approximation 
and the PRC is approximatively given by the discontinuous function (see \cite{Brown} for more details)
\begin{equation}
\label{PRC_homoclinic}
Z(\theta)=C\,\omega \exp \left(\frac{2\pi\lambda_u}{\omega}\right)\exp\left(-\lambda_u\frac{\theta}{\omega}\right)\,,
\end{equation}
where $C>0$ is a model-dependent constant. 

The PRC \eqref{PRC_homoclinic} is monotone decreasing and has a positive curvature. It follows that the oscillators close to a homoclinic bifurcation and interacting through a weak excitatory impulsive coupling ($0<K\ll 1$) satisfy the hypothesis of Proposition \ref{prop_dec}: they exponentially converge toward the asynchronous state, provided that the asynchronous state exists. The condition \eqref{cond2_E0} of Proposition \ref{prop_E0} is always satisfied with a weak coupling, so that the asynchronous state always exists. For a weak inhibitory coupling \mbox{($-1\ll K<0$)}, Proposition \ref{prop_sync} implies that the oscillators reach a synchronous state in finite time.

In the popular Morris-Lecar model \cite{Morris-Lecar}, which is among the most widely used conductance-based models in computational neuroscience, a homoclinic bifurcation can occur for low external currents \cite{Rinzel,Bif_Morris_Lecar}. The above results apply in this situation.
\hfill $\diamond$
\end{example}

\section{Conclusion}
\label{conclusion}

In this paper, we have studied the global behavior of infinite populations of monotone pulse-coupled oscillators. The behavior of the oscillators is dichotomic: either the oscillators achieve perfect synchrony (synchronous state) or the oscillators uniformly spread over $S^1(0,2\pi)$ (asynchronous state).

The infinite population is represented by a continuous density. In this framework, a necessary and sufficient condition ensures existence and uniqueness of the stationary density (asynchronous state).

For the global convergence analysis of the (nonlinear) transport equation of the density, we propose a Lyapunov function that is the total variation distance between quantile density functions. The Lyapunov function has two extreme values that correspond to the two steady-state behaviors of the system (the synchronous state and the asynchronous state). The stability results obtained for general phase oscillators are applied to particular models of importance (e.g. leaky integrate-and-fire model).

The main result of the paper stresses the importance of a $L^1$ distance (the total variation distance) to analyze a transport PDE under monotonicity assumptions on the PRC. However, the time evolution of the proposed Lyapunov function is no longer monotone when the monotonicity assumptions fail, even though the observed dichotomic behavior seems more general. This restriction raises interesting open questions about the generalization of the Lyapunov function and about the use of $L^1$ distances for a larger class of transport PDE's.

\appendix

The Lyapunov function is written as
\begin{eqnarray}
\mathcal{V}= \int_0^1 |q-q^*| \,d\varphi && = \int_{\mathcal{A}^+} (q-q^*) \,d\varphi + \int_{\mathcal{A}^-} (q^*-q) \, d\varphi \nonumber \\
&& \leq \int_{\mathcal{A}^+} (q-q_{\textrm{min}}) \,d\varphi + \int_{\mathcal{A}^-} (q^*-q_{\textrm{min}}) \, d\varphi\,, \label{V_less}
\end{eqnarray}
with $\mathcal{A}^+=\{\varphi\in[0,1]|q-q^*\geq0\}$ and $\mathcal{A}^-=\{\varphi\in[0,1]|q-q^*<0\}$.

Moreover, one has
\begin{equation}
\label{less_2pi}
\int_{\mathcal{A}^+} q \,d\varphi + \int_{\mathcal{A}^-} q_{\textrm{min}} \,d\varphi \leq \int_0^1 q \, d\varphi  =Q(1)-Q(0)=2\pi\,,
\end{equation}
and
\begin{equation}
\label{less_2pibis}
\int_{\mathcal{A}^-} q^* \,d\varphi + \int_{\mathcal{A}^+} q_{\textrm{min}} \,d\varphi  \leq \int_0^1 q^* \, d\varphi  =Q^*(1)-Q^*(0)=2\pi\,.
\end{equation}
Next, injecting \eqref{less_2pi} and \eqref{less_2pibis} in inequality (\ref{V_less}) yields
\begin{equation}
\nonumber
\mathcal{V}\leq 2\pi - 2 \int_{\mathcal{A}^-} q_{\textrm{min}} \, d\varphi + 2\pi - 2 \int_{\mathcal{A}^+} q_{\textrm{min}} \, d\varphi =4\pi-2 \,q_{\textrm{min}}\,.
\end{equation}

\bibliographystyle{siam}

\end{document}